\documentclass[12pt]{amsart}
\usepackage[letterpaper,margin=1.1in]{geometry}
\usepackage[utf8]{inputenc}
 \usepackage{tikz,esint,amsthm,amsmath,amstext,amssymb,amscd,pspicture,multicol,graphpap,graphics,graphicx,comment, enumerate,subfig,sidecap,wrapfig,color,pict2e}
 \usepackage{pgfplots}
\usepackage{hyperref}

\pgfplotsset{compat=1.11}
\usepgfplotslibrary{fillbetween}
\usetikzlibrary{intersections}
 \newcommand{\R}{\mathbb{R}}

\theoremstyle{plain}
\newtheorem{theorem}{Theorem}
\newtheorem{claim}[theorem]{Claim}

\newtheorem{corollary}[theorem]{Corollary}

\newtheorem{lemma}[theorem]{Lemma}

\newtheorem{conjecture}[theorem]{Conjecture}

\theoremstyle{definition}

\newtheorem{definition}[theorem]{Definition}
\newtheorem{remark}[theorem]{Remark}

\numberwithin{equation}{section}
\numberwithin{theorem}{section}

\makeatletter
\@namedef{subjclassname@2020}{\textup{2020} Mathematics Subject Classification}
\makeatother

\begin{document}

\title[Branch points for (almost-)Minimizers]{Branch points for (almost-)Minimizers of Two-Phase Free Boundary Problems}
\author[G. David, M. Engelstein, M. Smit Vega Garcia \& T. Toro]{Guy David, Max Engelstein, Mariana Smit Vega Garcia \& Tatiana Toro}
\thanks{
G.D. was partially supported by the ANR, programme blanc GEOMETRYA, ANR-12-BS01-0014, the European H2020 Grant GHAIA 777822, and the Simons Collaborations in MPS Grant 601941, GD.
M.E. has been partially supported by the NSF grant DMS 2000288. M.S.V.G. has been partially supported by the NSF grant DMS 2054282. 
T.T. was partially supported by the Craig McKibben \& Sarah Merner Professor in Mathematics, and by NSF grant DMS-1954545
}
\address{
Laboratoire de Math\'ematiques d'Orsay, Univ. Paris-Sud, CNRS, Universit\'e Paris-Saclay, 
91405 Orsay, France.}
\email{Guy.David@math.u-psud.fr}
\address{School of Mathematics,\\ University of Minnesota,\\ Minneapolis,
MN, 55455, USA.}
\email{mengelst@umn.edu}
\address{Western Washington University \\ Department of Mathematics \\
BH 230\\
Bellingham, WA 98225, USA}
\email{Mariana.SmitVegaGarcia@wwu.edu}
\address{Department of Mathematics\\ University of Washington\\ Box 354350\\ Seattle, WA 98195-4350}
\email{toro@uw.edu}
\date{\today}

\maketitle

\begin{abstract}
We study the existence and structure of branch points in two-phase free boundary problems. More precisely, we construct a family of minimizers to an Alt-Caffarelli-Friedman type functional whose free boundaries contain branch points in the strict interior of the domain. We also give an example showing that branch points in the free boundary of almost-minimizers of the same functional can have very little structure. This last example stands in contrast with recent results of De Philippis-Spolaor-Velichkov on the structure of branch points in the free boundary of stationary solutions. 
\end{abstract}

\section{Introduction}\label{s:Intro}

In this paper we study the structure of ``branch points" in the free boundary of minimizers of Alt-Caffarelli-Friedman type functionals (see \eqref{e:ACF} below). In particular, we show the existence of minimizers to the two-phase functional whose zero set contains an open subset (of positive measure) which stays far away from the fixed boundary of the domain. Relatedly, the free boundary of this minimizer also contains branch or cusp points (c.f. \eqref{e:BP}). We also show, in contrast with recent results for critical points to \eqref{e:ACF} in \cite{DeSpVePreprint}, that the set of branch points in the free boundary of almost-minimizers to \eqref{e:ACF} can have fractal like structure. 

Alt, Caffarelli and Friedman, in \cite{AlCaFr84}, gave the first rigorous mathematical treatment of the two-phase energy
\begin{equation}\label{e:ACF}
J_\Omega(u)=\int_{\Omega}|\nabla u|^2 dx+ \lambda_+^2|\{u>0\}\cap\Omega|+\lambda_-^2|\{u<0\}\cap\Omega|%+\lambda_0^2|\{u=0\}\cap\Omega|,
\end{equation}
where $\Omega\subset\R^n$ is a domain with locally Lipschitz boundary, and $\lambda_{\pm} > 0$\footnote{\cite{AlCaFr84} actually considered a related functional which also weighs the zero set, 
but \eqref{e:ACF} captures the essence of the more general functional, and for
almost minimizers the general case can be transformed to \eqref{e:ACF} through a standard change of notation.
}.

This is a two-phase analogue of the one-phase free boundary problem (also called the Bernoulli problem) studied in \cite{AlCa81}, first introduced to model the flow of two liquids in jets and cavities but later found to have applications to a variety of problems including eigenvalue optimization, c.f. \cite[Corollary 1.3]{DeSpVe21}. 

We say that $u$ is a {\it local minimizer} of $J$ in $\Omega$ if $J_D(u) \leq J_D(v)$ for all open $D$ with $\overline{D} \subset \Omega$ and all $v\in W^{1,2}(\Omega)$ with $u = v$ in $\Omega \backslash \overline{D}$. Alternatively, given some subset $S\subset \partial \Omega$ and continuous data $\varphi \in C(S)$ we say that $u$ minimizes $J_\Omega$ for the data $\varphi$ if $u = \varphi$ on $S$ and for any $v\in W^{1,2}(\Omega)$ with $v = \varphi$ on $S$ we have $J_\Omega(u) \leq J_\Omega(v)$. If the data $\varphi$ is not important, we simply say that $u$ is a minimizer of $J_\Omega$. We note that if $u$ minimizes $J_\Omega$ for some data $\varphi$, then $u$ is a local minimizer of $J$ in $\Omega$.

Given a (local) minimizer, $u$, of particular interest are the free boundaries, $\Gamma^\pm(u) = \partial \{\pm u > 0\}$. When $\Gamma^+\cap \Gamma^- = \emptyset$, each of $\Gamma^{\pm}$ is the free boundary of a minimizer to an associated one-phase problem and thus have well understood regularity (c.f. \cite{AlCa81}). On the other hand, when $\Gamma^+ = \Gamma^-$, the free boundary regularity is also well understood, first when $n=2$ in \cite{AlCaFr84}, and later by Caffarelli (\cite{Caf1, Caf2, Caf3}; see also the book \cite{CafBook}) and De Silva-Ferrari-Salsa (see e.g. \cite{DeFeSa14, DeFeSaReg} and the recent survey article \cite{DeFeSaSurvey}). Until recently, the only missing piece of the picture was the behavior of $\Gamma^{\pm}$ in neighborhoods where the two sets are not disjoint but also not identical. To be more precise, define the points in the intersection of $\Gamma^{\pm}$ as two-phase points;  $\Gamma_{\mathrm{TP}}(u) := \Gamma^+ \cap \Gamma^-$. Points which are in one of $\Gamma^{\pm}$ but not both are one-phase points;  $\Gamma_{\mathrm{OP}}(u) := \Gamma^+\cup \Gamma^- \setminus (\Gamma^+\cap \Gamma^-)$. It was a long open question how the free boundary behaved around branch points, that is, points around which the free boundary contains both one-phase points and two-phase points at every scale;  \begin{equation}\label{e:BP} \Gamma_{\mathrm{BP}}(u) := \Gamma_{\mathrm{TP}}(u) \cap \overline{\Gamma_{\mathrm{OP}}(u)}.\end{equation}

This open question was finally resolved in the recent work of De Philippis, Spolaor and Velichkov \cite{DeSpVe21} (see also \cite{SpVe19} when $n=2$):

\begin{theorem}(Main Theorem in \cite{DeSpVe21})\label{t:DSV}
Let $u$ be a (local) minimizer to the energy in \eqref{e:ACF} in $\Omega$ with $\lambda_{\pm} > 0$. Then for every $x_0 \in \Gamma^+ \cap \Gamma^-\cap \Omega$ there exists an $r_0 > 0$ such that both $\Gamma^+\cap B(x_0, r_0)$ and $\Gamma^-\cap B(x_0, r_0)$ are $C^{1,1/2}$-graphs.
\end{theorem}

We note that Theorem \ref{t:DSV} is most interesting around branch points, i.e. $x_0 \in \Gamma_{\mathrm{BP}(u)}\cap \Omega$. However, left open in \cite{DeSpVe21} is whether branch points actually exist in the strict interior of a domain, or more precisely, does there exist a minimizer $u$ in $\Omega$ such that $\Gamma_{\mathrm{BP}}(u)\cap \Omega \neq \emptyset$. Here we resolve that open question when $\lambda_+ = \lambda_- = 1$:

\begin{theorem}(Main Theorem)\label{t:main}
There exists a domain $\Omega \subset \mathbb R^2$ and a minimizer $u$ to $J_\Omega$ with $\lambda_+ = \lambda_- = 1$ such that $\Gamma_{\mathrm{BP}}(u)\cap \Omega \neq \emptyset$. Even stronger, there exists a ``pool" of zeroes: 
a (nonempty) connected component $\mathcal O$ of $\{u = 0\}$ 
such that 
$\overline{\mathcal O}  \subset \subset \Omega$ and $\partial \mathcal O \cap \Gamma^\pm \neq \emptyset$.

\end{theorem}

 Our tools are reminiscent of our previous studies on almost-minimizers with free boundary, c.f. \cite{DaTo15, DaEnTo19, DaEnSmTo21}. In particular, we carefully choose competitor functions and use ideas from harmonic analysis and geometric measure theory. We further remark that our construction can be extended to produce examples in dimensions $n\geq 3$, see Remark \ref{r:higherdim}.

\subsection{Comparison with other work on Branch Points}\label{ss:branch}
While we believe the question of whether branch points (or pools) in the free boundary of minimizers of \eqref{e:ACF} exist has been open until now, there has been substantial work on branch points for other related functionals and for ``critical points" of the functional \eqref{e:ACF}. 

In particular, branch points in the free boundaries of minimizers to a related vectorial problem were constructed by Spolaor and Velichkov in \cite{SpVe19}. Additionally, a related phenomena, when the free boundary of minimizers to a one-phase version of \eqref{e:ACF} comes into contact with the fixed boundary (i.e. $\partial \Omega$) resulting in branching like behavior, is well studied (e.g. \cite{ChSa19, DeSpVePreprint}). 

The only other work we are aware of regarding branch points in the free boundary of functions associated to the energy \eqref{e:ACF}, is the recent preprint \cite{DeSpVePreprint}. In this very nice work, the authors (amongst other things) construct an infinite family of {\it critical points} to the functional \eqref{e:ACF} when $n = 2$ using (quasi-)conformal mappings (see \cite[Theorem 1.8]{DeSpVePreprint}). Without being precise, we recall that critical points to \eqref{e:ACF} satisfy the associated Euler-Lagrange equations but do not necessarily (locally) minimize the functional in any domain (e.g. $u(x) = |x|$ is a critical point of $J_{\Omega}$ with $\lambda_+ = 1$ but not a (local) minimizer). 

\begin{remark}\label{r:compare}
To be explicit, we note that (none of) the results of \cite{DeSpVePreprint} either imply or are implied by our results here. In particular, our main theorem does not analyze the rate at which $\Gamma^+$ and $\Gamma^-$ come together at the cusp points, and thus does not produce examples with different rates. On the other hand, it is not clear whether the examples produced in \cite{DeSpVePreprint} are minimizers. 

Furthermore, the methods of proof are very different, in so far as \cite{DeSpVePreprint} draws an interesting connection with minimizers of a non-linear obstacle type problem and uses (quasi-)conformal maps in their construction. We construct the relevant boundary values and domains explicitly but do not have a closed formula for our minimizer. Rather, we use tools from harmonic analysis and geometric measure theory to constrain the behavior of the minimizer. In particular, our methods extend to producing examples in dimension $n > 2$; see Remark \ref{r:higherdim}, which presumably are out of reach of (quasi-)conformal methods.
\end{remark}

We are not aware of any prior work on ``pools" in the zero set of minimizers to \eqref{e:ACF}. We will limit ourselves to pointing out that it is easy to construct examples of minimizers whose zero set has non-empty interior but that some care is required to constrain this open component to the strict interior of the domain.

\subsection{Accumulation of Branch Points}\label{ss:Branch}

Also of interest in \cite{DeSpVePreprint} is the fact that for certain, symmetric (in a precise sense), critical points of \eqref{e:ACF} in two dimensions, the branch points in the free boundary are locally isolated in (c.f. \cite[Theorem 1.6(a)]{DeSpVePreprint}). In fact, in analogy with area-minimizing surfaces (see, e.g. \cite{Ch88, DeSpSp17, DeSpSp18, DeSpSp20, DeMaSpVa18}) one might conjecture the following:

\begin{conjecture}\label{conj:finite}
Let $u$ be a minimizer to \eqref{e:ACF} in some $\Omega \subset \mathbb R^n$. Then, for any $D\subset \subset \Omega$ the set $D\cap \Gamma_{\mathrm{BP}}(u)$ is locally contained in finitely many Lipschitz $(n-2)$-dimensional submanifolds. 
\end{conjecture}

In the second part of this paper we show that such a theorem fails for {\it almost-minimizers} of \eqref{e:ACF}. Recall that almost-minimizers minimize the energy \eqref{e:ACF} up to some noise:

\begin{definition}\label{d:am}
We say that $u$ is an almost-minimizer to \eqref{e:ACF} in $\Omega\subset \mathbb R^n$ if there is a $C > 0$ and an $\alpha \in (0, 1]$ such that for every ball $B$ with $\overline{B} \subset \Omega$ and for every $v\in W^{1,2}(\Omega)$ with $v = u$ in $\Omega \setminus \overline{B}$ we have $$J_B(u) \leq J_B(v) + Cr(B)^{n+\alpha}.$$
\end{definition} 

Almost-minimizers arise naturally in constrained optimization (and thus in eigenvalue-optimization problems see, e.g. \cite{MaTeVe17}). Almost-minimizers may not satisfy an Euler-Lagrange equation, but the work of the authors in \cite{DaTo15, DaEnTo19, DaEnSmTo21} show that the ``first order" regularity of almost-minimizers mimics that of minimizers; in particular, almost-minimizers are regular, Lipschitz continuous, non-degenerate, $C^{1,\beta}$ up to their free boundary, and, in the one-phase case, have free boundaries which are smooth up to a set of $\mathcal H^{n-1}$-measure zero. 

However in contrast to \cite[Theorem 1.6(a)]{DeSpVePreprint} and Conjecture \ref{conj:finite}, in this paper we prove that the set of branch points for almost-minimizers can be essentially arbitrary:

\begin{theorem}(Corollary to Theorem \ref{t:fbam})\label{t:ambranchpoints}
Let $E \subset \mathbb R^{n-1}$ be a compact set with no interior point.
 % Guy: just to be on the safe side about induced topology
%with $\partial E = E$. 
Embed $E$ into $\mathbb R^n$ so that 
% Guy: in fact it is a question: can you really say such that here? This shocks my french ear,
% but my french ear is not english.
% such that 
$E \subset \{(x', 0)\mid x' \in \mathbb R^{n-1}\}$ and let $R > 0$ be 
so large
% Guy : same story; probably delete, but I can't resist
% large enough such 
that $E \subset B(0, R/10)$. Then there exists an almost-minimizer, $u$, to $J_{B(0, R)}$ with $\lambda_+ = \lambda_- = 1$ such that $\Gamma_{\mathrm{BP}}(u) = E$. Furthermore, we can take $u$ to be such that $\Gamma^+$ is the reflection of $\Gamma^-$ across the hyper-plane $\{x_n =0\}$ (so that this solution is ``symmetric" in the sense of \cite{DeSpVePreprint}). 
\end{theorem}

For those familiar with almost-area minimizers, this theorem may seem trivial; indeed any graph of a $C^{1,\alpha}$ function is almost-area minimizing.
However, it is not the case that if $\partial \{u \neq 0\}$ is locally given by a (union of) smooth graphs and $u$ is smooth, except for jumps along $\partial \{u \neq 0\}$, then $u$ is an almost-minimizer to the energy in \eqref{e:ACF}. Indeed, almost-minimizers must satisfy additional non-degeneracy and boundedness conditions in addition to a condition on their normal derivative at $\partial \{u \neq 0\}$ (c.f. Lemma \ref{l:fbam} below). 

The essence of Theorem \ref{t:ambranchpoints} is that we are able to construct almost-minimizers to \eqref{e:ACF} whose free boundaries are given by the graphs of any two smooth functions $f^- \leq f^+$ over $\mathbb R^{n-1} \subset \mathbb R^n$. These almost-minimizers are given by a regularized distance (see \eqref{e:regdist} below) developed by the first author, with Feneuil and Mayboroda, to characterize the geometry of sets of high co-dimension using degenerate PDE (see, e.g. \cite{DaFeMa19}). This connection is surprising, but essentially is due to the fact that regularized distances satisfy the same growth conditions as almost-minimizers and, due to the results of \cite{DaEnMa21}, one can prescribe their normal derivatives along the set on which they vanish.

\section{Slice Minimizers and Uniform Lipschitz Continuity}\label{ss:prelim}

Let $I_N:=[-3N,3N]$, with $N$ very large, and consider $\overline{R}_N:=I_N\times [-1, 1]\subset \R^2$. To define our boundary conditions on  $I_N\times \{1\}$ and $I_N\times \{-1\}$,  we fix some $\alpha \in (0,1)$ small but universal ($\alpha = 1/10$ will do) and define
\begin{align*}
f_N(x)=\begin{cases} 1-\alpha, &\text{ if } |x|\le N\\
\frac{|x|(1+\alpha)}{N} -2\alpha &\text{ if } N\le |x|\le 2N\\
2 % , 
&\text{ if } 2N\le|x|\le 3N %.
\end{cases}
\end{align*}
(see Figure \ref{f:fndef}).

We then let \begin{equation}\label{e:undef}u_N = \mathrm{argmin}\{ J(u, \overline{R}_N)\mid 
u_N(x,\pm 1)=\pm f_N(x), x\in I_N\},\end{equation} that is, a minimizer to the Alt-Caffarelli-Friedman function with boundary values $\pm f_N(x)$. We should note that we are abusing the argmin notation slightly, as this minimizer is not necessarily unique, but we can pick any minimizer for the analysis below. We should also note that we are not prescribing Dirichlet data on the ``vertical" parts of the boundary. However, an existence and regularity theory for minimizers given ``partial Dirichlet data" exists (see e.g. \cite{AlCaFr84}) or we could prescribe data on $\{\pm 3N\} \times [-1,1]$ that simply linearly interpolates between $f_N(\pm 3N)$ and $-f_N(\pm 3N)$.

When the precise value of $N$ is not important we will suppress it from the notation. 

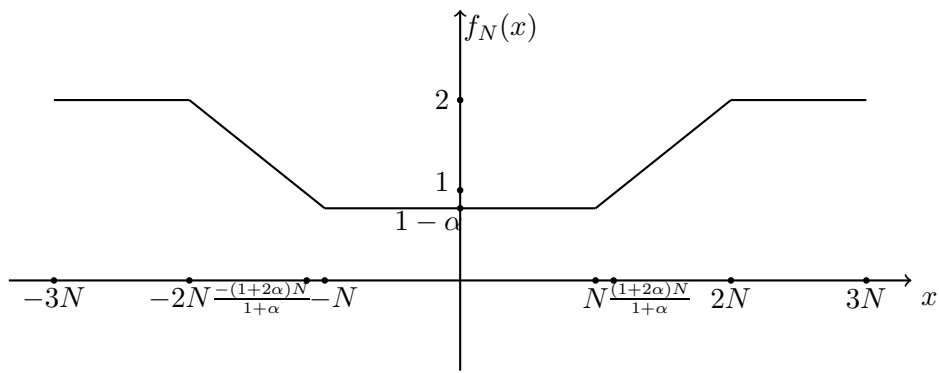
\begin{figure}
\begin{tikzpicture}[scale=1.2]
%\draw [thick, smooth] (0,0) circle (0.6cm);
%\draw [red, smooth, fill=black!40!green] (0.17,0.17) circle (0.27cm);
%\draw [blue, smooth] (-0.6,0) -- (0.6,0);
\draw [thick, smooth, ->] (-5,0) -- (5,0); 
\draw [thick, smooth, ->] (0,-1) -- (0,3);
\node at (1.5,-0.2) {\small $N$};
\node at (2.1, -0.2) {\tiny $\frac{(1+2\alpha)N}{1+\alpha}$};
\filldraw (1.70,0) circle (0.03cm);
\node at (-2.2, -0.2) {\tiny $\frac{-(1+2\alpha)N}{1+\alpha}$};
\filldraw (-1.70,0) circle (0.03cm);
\filldraw (1.5,0) circle (0.03cm);
\node at (-1.4,-0.2) {\small $-N$};
\filldraw (-1.5,0) circle (0.03cm);
\node at (3,-0.2) {\small $2N$};
\filldraw (3,0) circle (0.03cm);
\node at (-3.1,-0.2) {\small $-2N$};
\filldraw (-3,0) circle (0.03cm);
\node at (4.5,-0.2) {\small $3N$};
\filldraw (4.5,0) circle (0.03cm);
\node at (-4.5,-0.2) {\small $-3N$};
\filldraw (-4.5,0) circle (0.03cm);
\node at (5.2, -0.2) {\small $x$};
\node at (0.45, 2.8) {\small $f_N(x)$};
\draw [thick, smooth] (-1.5,0.8) -- (1.5,0.8);
\draw [thick, smooth] (1.5,0.8) -- (3,2);
\draw [thick, smooth] (3,2) -- (4.5,2);
\draw [thick, smooth] (-1.5,0.8) -- (-3,2);
\draw [thick, smooth] (-3,2) -- (-4.5,2);
\node at (-0.35, 0.65) {\small $1-\alpha$};
\filldraw (0,0.8) circle (0.03cm);
\node at (-0.2, 1.1) {\small $1$};
\filldraw (0,1) circle (0.03cm);
\node at (-0.2, 2) {\small $2$};
\filldraw (0,2) circle (0.03cm);
\end{tikzpicture}
\caption{Graph of $f_N$}
\label{f:fndef}
\end{figure}

In order to study $u_N$ we will introduce the ``slice minimizer", $v_N$, defined as follows: for each $x\in I$, $v_N(x, -)$ is the unique\footnote{Our computations in the Appendix show that this minimizer is unique} minimizer of the one-dimensional functional 
\begin{equation}\label{slice energy}
H_x(w):=\int_{-1}^1 (w'(y))^2dy + |\{w\neq 0\}\cap[-1,1]|
\end{equation}
under the constraint $v_N(x,\pm 1)=\pm f_N(x)$.

Notice first that for $x\in I$, the set $\{v_N(x, -) =0\}$ is an interval. Indeed, if $y_1$ and $y_2$ are the first/last points where $v_N$ vanishes, replacing $v$ by $0$ on the interval $[y_1,y_2]$ yields an admissible candidate $w$ with $\int_{-1}^1(w'(y))^2dy \le \int_{-1}^1(v')^2dy$ and $|\{w\neq 0\}|\le |\{v\neq 0\}|$, with strict inequality, unless $v_N\equiv 0$ on $[y_1,y_2]$. Moreover, $v_N$ is harmonic in the open set $\{v_N\neq 0\}\cap(-1,1)$, that is, $v_N$ is locally affine.
 A straightforward calculation, which we defer to the Appendix (c.f. Sections \ref{SS>1} and \ref{SS<1}), allows us to explicitly calculate the slice minimizer and its energy for each $x\in [-3N, 3N]$.

\begin{lemma}\label{slicedmin}
Let $v(y)=v_N(x,y)$ be the minimizer of $H_x$ 

with $v(x,\pm 1)=\pm f_N(x)$. 
\begin{enumerate}
    \item[Case 1:] When $f_N(x)\ge 1$, $v(y)=yf_N(x)$, for $y\in [-1,1].$ 
    
    Notice that in this case, $H_x(v)=2f_N(x)^2+2.$
    \item[Case 2:] When $f_N(x)<1$, $v(y)=\text{sgn}(y)(|y|-1+f_N(x))_+.$ 
    
    In this case, $H_x(v)=4f_N(x).$
\end{enumerate}
\end{lemma}

%\newpage
 % \begin{figure}
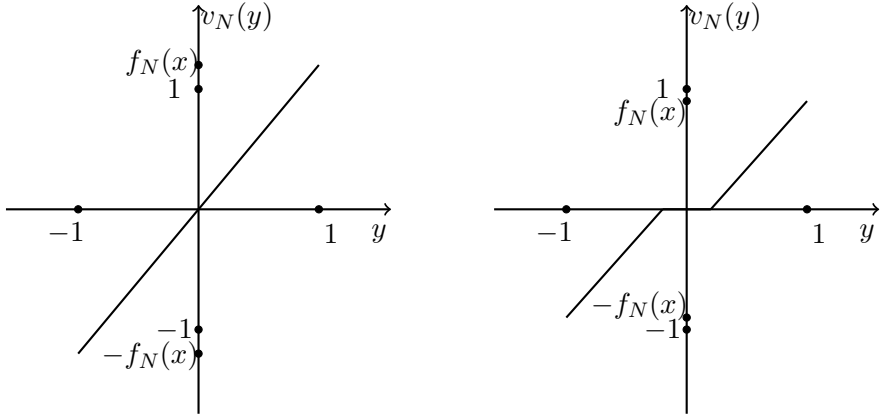
 \begin{center}
\begin{tikzpicture}[scale=1.6]
\draw [thick, smooth, ->] (-1.6,0) -- (1.6,0); 
\draw [thick, smooth, ->] (0,-1.7) -- (0,1.7);
\node at (-1.1,-0.2) {\small $-1$};
\node at (1.1, -0.2) {\small $1$};
\filldraw (1,0) circle (0.03cm);
\filldraw (-1,0) circle (0.03cm);
\node at (-0.2,1) {\small $1$};
\node at (-0.2,-1) {\small $-1$};
\filldraw (0,1) circle (0.03cm);
\filldraw (0,-1) circle (0.03cm);
\node at (-0.3,1.22) {\small $f_N(x)$};
\node at (-0.4,-1.22) {\small $-f_N(x)$};
\filldraw (0,1.2) circle (0.03cm);
\filldraw (0,-1.2) circle (0.03cm);
\node at (1.5, -0.2) {\small $y$};
\node at (0.32, 1.6) {\small $v_N(y)$};
\draw [thick, smooth] (-1,-1.2) -- (1,1.2);
\end{tikzpicture}
\hspace{10mm}
%\caption{Case 1: $f(x)\ge 1$}
\begin{tikzpicture}[scale=1.6]
\draw [thick, smooth, ->] (-1.6,0) -- (1.6,0); 
\draw [thick, smooth, ->] (0,-1.7) -- (0,1.7);
\node at (-1.1,-0.2) {\small $-1$};
\node at (1.1, -0.2) {\small $1$};
\filldraw (1,0) circle (0.03cm);
\filldraw (-1,0) circle (0.03cm);
\node at (-0.2,1) {\small $1$};
\node at (-0.2,-1) {\small $-1$};
\filldraw (0,1) circle (0.03cm);
\filldraw (0,-1) circle (0.03cm);
\node at (-0.3,0.8) {\small $f_N(x)$};
\node at (-0.39,-0.8) {\small $-f_N(x)$};
\filldraw (0,0.9) circle (0.03cm);
\filldraw (0,-0.9) circle (0.03cm);
\node at (1.5, -0.2) {\small $y$};
\node at (0.32, 1.6) {\small $v_N(y)$};
\draw [thick, smooth] (-1,-0.9) -- (-0.2,0);
\draw [thick, smooth] (-0.2,0) -- (0.2,0);
\draw [thick, smooth] (0.2,0) -- (1,0.9);
\end{tikzpicture}
\end{center}
%\caption{Case 2: $f(x)<1$}
%\end{figure}

A crucial observation is that even though we built $v_N$ ``slice by slice", its $x$-derivative still has small $L^2$-norm.

\begin{lemma}\label{l:dvdxnotbig} We have
\begin{equation}\label{dvdx}
\iint_{\overline{R}_N}\left|\frac{\partial v_N}{\partial x}\right|^2\le \frac{16}{N}.
\end{equation}
\end{lemma}

\begin{proof}
Notice that $v_N(x,y)$ is not harmonic in $\{v\neq 0\}$, hence $v_N$ is not a minimizer of $J(-, R_N)$. However, $\frac{\partial v_N}{\partial x}$ exists a.e. in $(-3N,3N)\times (-1,1)$ and, where it exists, $\left|\frac{\partial v_N}{\partial x}\right|\le |f_N'(x)|\le\frac{1+\alpha}{N}\le\frac{2}{N}.$ 

Since $f(x)\equiv 1-\alpha$ for $|x|<N$ and $f(x)\equiv 2$ for $2N<|x|<3N$, $\left|\frac{\partial v_N}{\partial x}\right|=0$ on $(-N,N)\times (-1,1)$, on $(2N,3N)\times (-1,1)$ and on $(-3N,-2N)\times(-1,1)$. Consequently,
\[
\iint_{\overline{R}_N}\left|\frac{\partial v_N}{\partial x}\right|^2=\int_{-2N}^{-N}\int_{-1}^1\left|\frac{\partial v_N}{\partial x}\right|^2+\int_{N}^{2N}\int_{-1}^1\left|\frac{\partial v_N}{\partial x}\right|^2\le 
\left(\frac{2}{N}\right)^2(2N)\leq \frac{16}{N}.
 % Guy: I think it is only 2 and not 4. Too small to bother fixing them all.
% \left(\frac{2}{N}\right)^2(4N)=\frac{16}{N}.
\]
\end{proof}

Using the fact that $v_N$ has smaller energy than $u_N$ ``slice by slice", but larger energy overall, we can transfer \eqref{dvdx} to $u_N$:

\begin{lemma}\label{l:xderivbound} Any $u_N$ which minimizes $J$ in $\overline{R}_N$ with $u_N(x,\pm 1)=\pm f_N(x)$, satisfies
\begin{equation}\label{ux} \iint_{\overline{R}_N}\left|\frac{\partial u_N}{\partial x}\right|^2\le \frac{16}{N}.
\end{equation}
\end{lemma}
\begin{proof}
Since $u_N$ is a minimizer of $J$,
\begin{equation}\label{uu}
\iint_{\overline{R}_N} |\nabla u_N|^2+|\{u_N\neq 0\}|\le \iint_{\overline{R}_N}|\nabla v_N|^2+|\{v_N\neq 0\}|.
\end{equation}
Moreover, for every $x\in[-3N,3N]$ fixed, $v_N(x,\cdot)$ is a minimizer of $H_x$, hence
\[
\int_{-1}^1\left|\frac{\partial v_N}{\partial y}\right|^2dy+|\{v_N(x,\cdot)\neq 0\}|\le\int_{-1}^1\left|\frac{\partial u_N}{\partial y}\right|^2dy+|\{u_N(x,\cdot)\neq 0\}|.
\]
Integrating the last inequality 
on % between 
$[-3N,3N]$ leads to
\begin{equation}\label{vv}
\iint_{\overline{R}_N}\left|\frac{\partial v_N}{\partial y}\right|^2+|\{v_N \neq 0\}|\le\iint_{\overline{R}_N}\left|\frac{\partial u_N}{\partial y}\right|^2+|\{u_N \neq 0\}|.
\end{equation}
Combining \eqref{uu} and \eqref{vv} we conclude
\[
\iint_{\overline{R}_N} \left|\frac{\partial u_N}{\partial x}\right|^2\le\iint_R \left|\frac{\partial v_N}{\partial x}\right|^2.
\]
\end{proof}

Our next goal is to prove a uniform Lipschitz bound on $u_N$. To do so, we will compare it with $v_N$. Since the latter minimizes the energy on each slice, it will be convenient to integrate this ``slice-by-slice" energy across all values of $x$:

\begin{definition}\label{slicedtotal}
Define the ``total sliced energies''
of a function $w$ by

\begin{equation}
S(w)=\int_{-3N}^{3N}\int_{-1}^1\left|\frac{\partial w}{\partial y}\right|^2dydx+|\{w\neq 0\}|=\int_{-3N}^{3N}H_x(w(x,\cdot))dx,
\end{equation}
and with $Q = [a,b] \times [-1, 1]\subset \overline{R}_N$,
\begin{equation}
    S_Q(w)=\iint_Q\left|\frac{\partial w}{\partial y}\right|^2dydx+|\{w\neq 0\}\cap Q|=\int_{a}^bH_x(w(x,\cdot))dx.
\end{equation}
\end{definition}

The following lemma encapsulates the fact that $u_N$ is a minimizer and $v_N$ is a slice minimizer, written in the language of total slice energy: 

\begin{lemma} We have
\begin{equation}\label{m1}
S(v_N)\le S(u_N)\le J(u_N)\le S(v_N)+\frac{16}{N}.
\end{equation}
\end{lemma}
\begin{proof}
Let $w\in W^{1,2}(\overline{R}_N)$. Notice that 
\begin{equation}\label{HJ}
S(w)= \iint_{\overline{R}_N}\left|\frac{\partial w}{\partial y}\right|^2+|\{w\neq 0\}|\le J(w), \quad 
S(w)+\iint_{\overline{R}_N}\left|\frac{\partial w}{\partial x}\right|^2=J(w).
\end{equation}
Using \eqref{HJ} and the fact that $H_x(v_N(x,\cdot))\le H_x(u_N(x,\cdot))$ a.e., we obtain
\[
S(v_N)\le S(u_N)\le J(u_N).
\]
Combining this with the equality in \eqref{HJ} and with \eqref{dvdx}, we conclude that if $u_N$ is a minimizer of $J$, then
\[
S(v_N)\le S(u_N)\le J(u_N)\le J(v_N)=S(v_N)+\iint_{\overline{R}_N}\left|\frac{\partial v_N}{\partial x}\right|^2\le S(v_N)+\frac{16}{N}.
\]
\end{proof}

We can also localize these estimates to $Q=[a,b]\times[-1,1]$:

\begin{lemma}\label{L:slicedQ} We also have
\begin{equation}\label{slicedQ}
S_Q(u_N)\le S_Q(v_N)+\frac{16}{N}.
\end{equation}
\end{lemma}

\begin{proof} Given a subset $X\subset[-3N,3N]\times[-1,1]$ and defining
\[
S_X(u_N)=\iint_X\left|\frac{\partial u_N}{\partial y}\right|^2dxdy+|\{u_N\neq 0\}\cap X|,
\]
we obtain $S(u_N)=S_Q(u_N)+S_{{\overline{R}_N}\setminus Q}(u_N)$. 
If we had $S_Q(u_N)>S_Q(v_N)+\frac{16}{N}$, then 
\begin{equation}\label{contra}
S(u_N)>S_Q(v_N)+\frac{16}{N}+
S_{\overline{R}_N\setminus Q}(u_N).
\end{equation}
Since for a.e. $x\in[-3N,3N]$ we have $H_x(v_N(x,\cdot))\le H_x(u_N(x,\cdot))$, integrating this inequality we obtain
$S_{{\overline{R}_N}\setminus Q}(v_N)\leq S_{{\overline{R}_N}\setminus Q}(u_N)$. Together with \eqref{contra}, this gives
\[
S(u_N)>S_Q(v_N)+\frac{16}{N}+S_{{\overline{R}_N}\setminus Q}(v_N)=S(v_N)+\frac{16}{N},
\]
contradicting \eqref{m1}.
\end{proof}

With these estimates we are almost ready to prove 
the (uniform) Lipschitz continuity of the $u_N$ on compact subsets of ${\overline{R}_N}$. We introduce the following notation, by analogy with Definition \ref{slicedtotal}, 
for $Q = [a,b] \times [-1,1]$ 
\begin{equation}\label{e:localizedenergy} 
J_Q(w) := \iint_{Q} |\nabla w|^2\, dA + |Q \cap \{w\neq 0\}|.
\end{equation}

Our first result is an immediate corollary of Lemma \ref{L:slicedQ} and \eqref{ux}:

\begin{corollary}\label{c:uniformenergybound}
There exists a constant $C_0 > 0$ such that for any $N > 0$ and any $Q = [a,b]\times [-1,1] \subset \overline{R}_N$ we have $$J_Q(u_N) \leq C_0|b-a| + \frac{32}{N}. % added .
$$
\end{corollary}

\begin{proof}
From Lemma \ref{L:slicedQ} we have that $S_Q(u_N) \leq S_Q(v_N) + \frac{16}{N}$ and from \eqref{ux} we have $\iint_Q |\partial_x u_N|^2\, dA \leq \frac{16}{N}$. Putting this together we get that $$J_Q(u_N) \leq S_Q(v_N) + \frac{32}{N}.$$ 

Thus it suffices to show that $S_Q(v_N)$ grows proportionally to $|b-a|$ with a constant of proportionality independent of $a,b, N$. Indeed, by Lemma \ref{slicedmin} $H_x(v_N) \leq \max\{2f_N(x)^2 + 2, 4f_N(x)\} \leq 10$. Integrating that across $[a,b]$ gives the desired result. 
\end{proof}

From here we can conclude the main result of this section, the uniform Lipschitz continuity:

\begin{theorem}\label{t:uniformlip}
For $0 < \delta < 1$, we can find constants $L \geq1$ and $N_0 > 0$ such that
$$\|u_N\|_{\mathrm{Lip}(\Omega_N)} \leq L$$
for $N \geq N_0$,
where $\Omega_N = (-3N+\delta, 3N-\delta) \times (-1+\delta, 1-\delta)$.
\end{theorem}

\begin{proof}
As $u_N$ is a minimizer, we can apply \cite[Theorem 8.1]{DaTo15} (c.f. the discussion at the bottom of page 504 in \cite{DaTo15}) which gives a Lipschitz bounds on almost-minimizers depending only on the distance from the boundary and the $L^2$-norm of the gradient (see also \cite[Remark 2.2]{DaEnTo19}). Corollary \ref{c:uniformenergybound} gives uniform bounds on the $L^2$-norm of the gradient of $u_N$ inside of any rectangle and a covering argument finishes the proof.
\end{proof}

\section{Existence of a Zero set}\label{s:poolexists}

The goals of this section are two fold. First to prove that $|\{u_N = 0\}| > 0$ 
(which will follow from Lemma \ref{l:doesntdiplow})
and second to show that the set $\{u_N = 0\}$ does not get too close to the boundary of $R_N$ (Lemma \ref{l:nogostripone} and Corollary \ref{c:reifenbergflat}).

Let us first describe the zero set of $v_N$, the ``slice minimizer":

\begin{lemma} The following holds regarding the set $\{v_N(x,y)=0\}$:
\begin{itemize}
    \item When $|x|\ge \frac{(1+2\alpha)N}{1+\alpha}$, $v_N(x,y)=0$ only when $y=0$. 
\item When $|x|\le N$, $v_N(x,y)=0$ for $|y|\le \alpha$. 

\item When $N<|x|<\frac{(1+2\alpha)N}{1+\alpha}$, $v_N(x,y)=0$ when $|y|\le 2\alpha+1-\frac{|x|(1+\alpha)}{N}$.
\end{itemize}

\end{lemma}

\begin{proof} The result follows from the following simple observations:

\begin{itemize}
    \item When $|x|\ge \frac{(1+2\alpha)N}{1+\alpha}$, $f_N(x)\ge 1$. In this case $v_N(x,y)=yf_N(x)$.
    \item When $|x|\le N$, $f_N(x)\equiv 1-\alpha$, and $v_N(x,y)=\text{sgn}(y)(|y|-\alpha)_+$.
    \item In the remaining interval, $f_N(x)=\frac{|x|(1+\alpha)}{N}-2\alpha$ and $v_N(x,y)=\text{sgn}(y)(|y|-1+\frac{|x|(1+\alpha)}{N}-2\alpha)_+$.
\end{itemize}
\end{proof}

%\begin{figure}
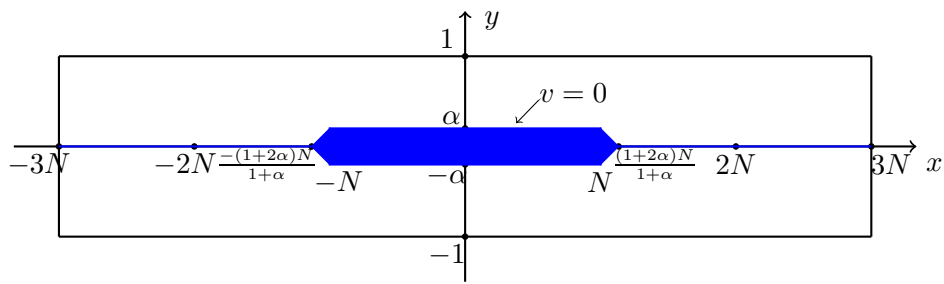
\begin{center}
\begin{tikzpicture}[scale=1.2]
%\draw [thick, smooth] (0,0) circle (0.6cm);
%\draw [red, smooth, fill=black!40!green] (0.17,0.17) circle (0.27cm);
%\draw [blue, smooth] (-0.6,0) -- (0.6,0);
\draw [thick, smooth, ->] (-5,0) -- (5,0); 
\draw [thick, smooth, ->] (0,-1.5) -- (0,1.5);
\node at (1.5,-0.4) {\small $N$};
\node at (2.1, -0.2) {\tiny $\frac{(1+2\alpha)N}{1+\alpha}$};
\filldraw (1.70,0) circle (0.03cm);
\node at (-2.2, -0.2) {\tiny $\frac{-(1+2\alpha)N}{1+\alpha}$};
\filldraw (-1.70,0) circle (0.03cm);
\filldraw (1.5,0) circle (0.03cm);
\node at (-1.4,-0.4) {\small $-N$};
\filldraw (-1.5,0) circle (0.03cm);
\node at (3,-0.2) {\small $2N$};
\filldraw (3,0) circle (0.03cm);
\node at (-3.1,-0.2) {\small $-2N$};
\filldraw (-3,0) circle (0.03cm);
\node at (4.72,-0.2) {\small $3N$};
\filldraw (4.5,0) circle (0.03cm);
\node at (-4.72,-0.2) {\small $-3N$};
\filldraw (-4.5,0) circle (0.03cm);
\node at (5.2, -0.2) {\small $x$};
\node at (0.3, 1.4) {\small $y$};
\node at (-0.2, 1.2) {\small $1$};
\node at (-0.2, -1.2) {\small $-1$};
\node at (-0.15,0.32) {\small $\alpha$};
\node at (-0.2,-0.32) {\small $-\alpha$};
\filldraw (0,0.2) circle (0.03cm);
\filldraw (0,-0.2) circle (0.03cm);
\filldraw (0,1) circle (0.03cm);
\filldraw (0,-1) circle (0.03cm);
\draw [thick, smooth] (-4.5,1) -- (4.5,1); 
\draw [thick, smooth] (-4.5,-1) -- (4.5,-1);
\draw [thick, smooth] (-4.5,-1) -- (-4.5,1); 
\draw [thick, smooth] (4.5,-1) -- (4.5,1); 
\draw [thick, blue, smooth, fill=blue] (-1.5,-0.2) rectangle (1.5,0.2);
\filldraw[draw=blue, fill=blue] (1.5,0.2) -- (1.7,0) -- (1.5,-0.2) -- (1.5,0.2) -- cycle;
\filldraw[draw=blue, fill=blue] (-1.5,0.2) -- (-1.7,0) -- (-1.5,-0.2) -- (-1.5,0.2) -- cycle;
\draw [thick, blue] (-4.5,0) -- (-1.7,0); 
\draw [thick, blue] (1.7,0) -- (4.5,0);
\draw [-<] (0.8,0.5) to[out=45, in =230] (0.6,0.3) ;
\node at (1.2,0.6) {\small $v=0$};
\end{tikzpicture}
\end{center}
%\end{figure}

We expect a minimizer $u_N$ of $J$, taken among all functions $w\in W^{1,2}({\overline{R}_N})$ with $w(x,\pm 1)=\pm f_N(x)$,\footnote{Recall we do not need to specify the data on the ``vertical" components of the boundary} to look similar to $v_N$. In particular, we want to extract information about its zero set, and prove that $\{u_N=0\}$ has a ``pool'' close to $0$.

Before we can prove this closeness, we need to observe that our minimizer is nice  on 
``most" of the vertical slices.

\begin{definition} 
Let $X_0\subset I$ be the smallest set such that $x\notin X_0$ implies that $u_N(x,\cdot)\in W^{1,2}([-1,1])$ and $\lim\limits_{y\rightarrow \pm 1}u_N(x,y)=\pm f_N(x)$. 

Since $u_N \in W^{1,2}(R_N)$ we note that $X_0$ has measure zero. 
\end{definition}

We now show that the zero set of $u_N$ does not get too close to the ``top" or ``bottom" of the rectangle. We start by showing that if $u_N$ is small close to the top or bottom of the rectangle, then that slice has large energy:

\begin{lemma}\label{ifusmallenergybig}
Let $\varepsilon \in (0,1), \delta \in (0,1/2)$ and assume that $|u_N(x, y)| < \delta$ for some $x\in [-3N, 3N]\backslash X_0$ and some $y$ with $1-|y| < \varepsilon$. Then $H_x(u_N(x,-)) \geq \frac{(f_N(x) - \delta)^2}{\varepsilon}$. 

In particular, if 
$x \in I \setminus X_0$ and there exists $y$ such that $1-|y| < \frac{1}{44}$ 
and $|u_N(x,y)| < \frac{1}{4}$, then $H_x(u_N(x, -)) \geq H_x(v_N(x,-)) + 1$.
\end{lemma}

\begin{proof}
Without loss of generality, assume $u_N(x,-)$ is both equal to $f_N(x)$ and $\delta$ on an interval of length $\varepsilon$. The lowest energy way to do this is assuming these values are achieved at the endpoints of the interval and that $u_N$ interpolates between them linearly. Thus $H_x(u_N(x, -)) \geq \frac{(f_N(x) - \delta)^2}{\varepsilon}$. The second result follows from the first once we remember that $f_N(x) \geq \frac{3}{4}$ for all $x$, and $H_x(v_N(x,-)) \leq 10$ for all $x$. 
\end{proof}

We are now ready to show the existence of a strip near the top and bottom of $R$, on which
 $u$ cannot vanish. We actually show something stronger which is that $u$ is quantitatively large in this strip:

\begin{lemma}\label{l:nogostripone}
Let $\delta > 0$ and set 
\begin{equation}\label{3a1}
R_\pm = \left\{ (x,r) \in R_N \, ; \,  0 \leq |x| \leq 3N-\delta \text{ and } 
\pm y \in \left[1-\frac{1}{44}, 1-\frac{1}{88}\right] \right\}
\end{equation}
(the two blue zones in the picture).
 Then there exists an $N_0 = N_0(\delta) > 1$ such that if $N > N_0$
then
\begin{equation}\label{3a2}
u_N(x,y) \geq 1/8 \ \text{ on $R_+$ and $u_N(x,y) \leq -1/8$ on } R_-.
\end{equation}
\end{lemma}

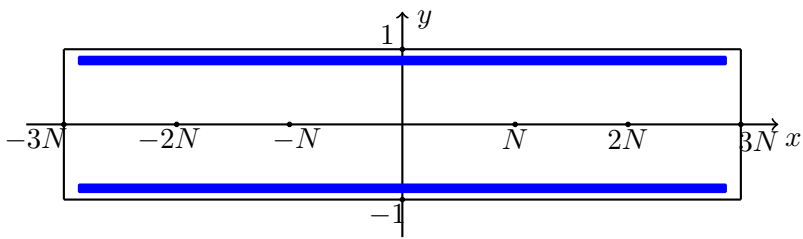
\begin{center}
\begin{tikzpicture}[scale=1.0]
\draw [thick, smooth, ->] (-5,0) -- (5,0); 
\draw [thick, smooth, ->] (0,-1.5) -- (0,1.5);
\node at (1.5,-0.2) {\small $N$};
\filldraw (1.5,0) circle (0.03cm);
\node at (-1.4,-0.2) {\small $-N$};
\filldraw (-1.5,0) circle (0.03cm);
\node at (3,-0.2) {\small $2N$};
\filldraw (3,0) circle (0.03cm);
\node at (-3.1,-0.2) {\small $-2N$};
\filldraw (-3,0) circle (0.03cm);
\node at (4.74,-0.22) {\small $3N$};
\filldraw (4.5,0) circle (0.03cm);
\node at (-4.87,-0.2) {\small $-3N$};
\filldraw (-4.5,0) circle (0.03cm);
\node at (5.2, -0.2) {\small $x$};
\node at (0.3, 1.4) {\small $y$};
\node at (-0.2, 1.2) {\small $1$};
\node at (-0.2, -1.2) {\small $-1$};
\filldraw (0,1) circle (0.03cm);
\filldraw (0,-1) circle (0.03cm);
\draw [thick, smooth] (-4.5,1) -- (4.5,1); 
\draw [thick, blue, smooth] (-4.3,0.9) -- (4.3,0.9); 
\draw [thick, blue, smooth] (-4.3,0.8) -- (4.3,0.8); 
\draw [thick, blue, smooth] (-4.3,0.8) -- (-4.3,0.9);
\draw [thick, blue, smooth] (4.3,0.8) -- (4.3,0.9);
\filldraw[draw=blue, fill=blue] (-4.3,0.9) rectangle (4.3,0.8);
\draw [thick, blue, smooth] (-4.3,-0.9) -- (4.3,-0.9); 
\draw [thick, blue, smooth] (-4.3,-0.8) -- (4.3,-0.8); 
\draw [thick, blue, smooth] (-4.3,-0.8) -- (-4.3,-0.9);
\draw [thick, blue, smooth] (4.3,-0.8) -- (4.3,-0.9);
\filldraw[draw=blue, fill=blue] (-4.3,-0.9) rectangle (4.3,-0.8);
\draw [thick, smooth] (-4.5,-1) -- (4.5,-1);
\draw [thick, smooth] (-4.5,-1) -- (-4.5,1); 
\draw [thick, smooth] (4.5,-1) -- (4.5,1); 
\end{tikzpicture}
\end{center}

\begin{proof}
Fix $\delta > 0$. Recall from Theorem \ref{t:uniformlip}
that there exists $L > 0$ (independent of $N$ but dependent on $\delta > 0$) such that if 
\[
\Omega = \left[-3N + 
\delta, 3N-
\delta\right] \times \left[-1+\frac{1}{100}, 1-\frac{1}{100}\right],
\]
then $\|u_N\|_{\mathrm{Lip}(\Omega)} \leq L$.

We first check that $|u_N(x_0,y_0)| < \frac18$ on $R_\pm$. Note that if
$|u_N(x_0,y_0)| < \frac18$ for some $(x_0,y_0) \in [-3N + \delta, 3N-\delta]\times ([-1 +\frac{1}{88}, -1+\frac{1}{44}]\cup [1-\frac{1}{44}, 1-\frac{1}{88}])$, then by Lipschitz continuity there exists an interval % of 
$I \subset [-3N + \delta, 3N-\delta]$ of length $\frac{1}{8L}$ such that $|u(x,y_0)| < \frac{1}{4}$ for all $x\in I$. 

We apply Lemma \ref{ifusmallenergybig} to conclude that for almost every $x\in I$ we have $H_x(u_N(x,-)) \geq H_x(v_N(x,-)) + 1$. If $Q = I \times [-1,1]$, then this implies that $S_Q(u_N) \geq S_Q(v_N) + |I| = S_Q(v_N) + \frac{1}{8L}.$ Of course, this contradicts \eqref{slicedQ} as long as $\frac{16}{N} < \frac{1}{8L}\Leftrightarrow 128L < N$.

Now we check that $u$ has the right sign on $R_\pm$. Suppose for instance 
that $u(x,y) \leq -1/8$ somewhere on $R_+$. Since $u$ is continuous, $u(x,y) \leq  -1/8$ 
everywhere on $R_+$. Then for all $x \in \left[-3N + \delta, 3N-\delta\right] \setminus X_0$,
$u(x,-)$ is a Sobolev function that goes from $-1/8$ to at least $1/4$ in an interval of
length at most $1/88$, a direct computation shows that 
$H_x(u_N(x,-)) \geq \frac{3}{8} \times 88 \geq 33$, and we reach a contradiction as above.
\end{proof}

Now we show that $\{u_N = 0\}$ must be contained in a strip around $\{y=0\}$ when $|x| < N$. Actually we prove something more precise: 

\begin{lemma}\label{l:doesntdiplow}
There exists an $N > 1$ large enough such if $|x| < N-1$, and $u_N(x,y) > 0$, then $y > \alpha/8$. Similarly, if $u_N(x,y) < 0$ then $y < -\alpha/8$. 
\end{lemma}

\begin{proof}
Assuming by contradiction this were not the case, without loss of generality there would exist a point in $(x_0, y_0)$ with $|x_0| < N-1, y_0 \leq \alpha/8$ and $u_N(x_0, y_0) > 0$. 

By continuity of $u_N$, the connected component of $\{u_N > 0\}$ containing $(x_0, y_0)$ must be seperated from $\{y = -1\}$ by $\{u_N = 0\}$. This implies that there is a connected subset of $\{(x,y)\mid |x| \leq N, u_N(x,y) = 0\}$ which touches the sets $\{x = N\}$ and $\{x= -N\}$ and which separates $(x_0, y_0)$ from $\{y =-1\}$. By Lemma \ref{l:nogostripone} this connected component cannot intersect the set $\{(x,y) \mid |x| < N, -1+1/44 < y < -1+1/88\}$. If $y_0 < -1+1/44$ then this connected component lies below $y = -1+1/88$ and the length of its projection onto the $x$-axis is at least $2N$. Such a configuration has too much energy by Lemma \ref{ifusmallenergybig} and thus we can assume $y_0 > -1+1/44$.

We can also assume that $x_0 \notin X_0$ (since $\{u_N > 0\}$ is open), so by continuity of $u_N$ 
on the slice $\{x = x_0\}$ and Lemma \ref{l:nogostripone} 
there exists a point $(x_0, \tilde{y})$ with $-1+1/44<\tilde{y} < y_0$ and $u_N(x_0, \tilde{y}) = 0$ and $(x_0, \tilde{y}) \in \partial \{u_N > 0\}$.

We note that $\{u_N > 0\}\cap ([-N, N]\times [-1+1/88, 1-1/88])$ is a locally NTA domain, uniformly in $N$ (i.e. for any $K \subset \subset ([-N, N]\times [-1+1/88, 1-1/88])$, $\{u_N >0\}$ satisfies the corkscrew conditions at $Q\in \partial \{u_N > 0\}\cap K$ with constants and at scales that depend only on $K$ not $N$ c.f. \cite[Theorem 2.3]{DaEnTo19})). In particular, there exists a point $(x_1, y_1) \in \{u_N > 0\}$ such that $\|(x_1, y_1) - (x_0, \tilde{y})\| \leq r_0 = r_0(K) \leq \alpha/8$ and $\mathrm{dist}((x_1, y_1), \{u_N \leq 0\}) \geq r_0/M$ for some $M > 1$ (where both $r_0, M$ are independent of $N$ large).

 By the non-degeneracy of $u_N$ (c.f. \cite[Lemma 3.4]{AlCa81}) and the Lipschitz continuity of $u_N$, Theorem \ref{t:uniformlip}, there exists a constant $C > 1$ (again uniform for large $N$) such that $u_N \geq r_0/C$ in the ball $B((x_1, y_1), r_0/(3CL))$ (where $L$ is the Lipschitz constant).

Therefore, there exists an interval $I \subset [-N, N]$ of length $2r_0/(3CL)$ such that for each $x\in I$ there exists a $y < \alpha/2$ with $u_N(x, y) > r_0/C$. Invoking the computations of Section~\ref{SS:X2} % added a tilde
(c.f. Claim \ref{cl:energybound}), we get that $H_x(u_N(x,-)) \geq H_x(v_N(x, -)) + \eta$ for every $x\in I \backslash X_0$, where $\eta = \eta(\alpha, C, r_0) > 0$ is independent of $N$. 

Let $Q = I \times [-1,1]$; % and 
we get (from \eqref{slicedQ}) $$S_Q(v_N) + \frac{16}{N} \geq S_Q(u_N) \geq S_Q(v_N) + \mathcal H^1(I)\eta.$$ This gives a contradiction if $N > 0$ is large enough (since $|I| = 2r_0/(3CL), \eta > 0$ are independent of $N$). 
\end{proof}

So we have a good control on where $\{u_N = 0\}$ is in the central region.
Before we end this section, it behooves us to refine the result of Lemma \ref{ifusmallenergybig}, with the goal of showing that when $f_N(x) \geq 1$, we can actually confine $\{u_N = 0\}$ to an arbitrarily thin strip around the line $\{y=0\}$.
This will be used in the next section to show that there are no one-phase points
on the sides of $R_N$. We first estimate the difference between $v_N(x,-)$ and near minimizers for $H_x$.

\begin{lemma}\label{error} 
Let $x\in (-3N, 3N)$ be % added be
such that $f_N(x) \geq 1$ and let $\varepsilon \in (0,1)$. Let $w\in W^{1,2}([-1,1])$, with $w(\pm 1) = \pm f_N(x)$ and assume $H_x(w)\leq H_x(v_N(x,-)) + \varepsilon$. Then, there exists a $C > 0$ (uniform over the choice of $x, \varepsilon$ above), such that
\[
||w-v_N||_{L^{\infty}([-1,1])}\le C\varepsilon^{1/2}.
\]
\end{lemma}

\begin{proof} 
Recall that for $x$ as in the statement of Lemma \ref{error}, we have $v_N(x,y) = yf_N(x)$. Our plan is to first modify $w$, reducing energy, and show the desired inequality for $w$, then estimate the $L^\infty$ distance between the original $w$ and our modified functional. We know that $w\in W^{1,2}([-1,1])$ so by Sobolev embedding we have that $w$ is H\"older continuous and thus $$-1 < \alpha := \inf\{t\mid w(t) =0\} \leq \sup\{t\mid w(t) = 0\} =: \beta < 1.$$ We construct $$\hat{w}(t) = \begin{cases} 0& t\in [a,b]\\ \frac{f_N(x)(t-a)}{1+a}& t\in [-1, a]\\
\frac{f_N(x)(t-b)}{1-b}& t\in [b, 1].\end{cases}$$ We observe that $H_x(w) \geq H_x(\hat{w})$ (as we have enlarged the zero set and minimized Dirichlet energy where $w$ is positive).

We can compute that $$H_x(\hat{w}) = f_N^2(x)\left(\frac{1}{1-b} + \frac{1}{a+1}\right) + 2-(b-a).$$ Recall that $H_x(v_N(x, -)) = 2f_N^2(x) + 2$ and rewrite $$\begin{aligned} H_x(\hat{w}) - H_x(v_N(x,-)) =& f_N^2(x) \left(\frac{1}{1-b} + \frac{1}{a+1} -2\right) -(b-a)\\ \geq& \left(\frac{1}{1-b} + \frac{1}{a+1} -2\right) -(b-a) =: F(a,b),\end{aligned}$$ where the last inequality follows because $f_N(x) \geq 1$ in the salient range and $\frac{1}{1-b} + \frac{1}{a+1} - 2 \geq 0$ as long as $1 > b \geq a >-1$.

We compute that $F(0,0) = 0, \nabla F(0,0) = (0,0)$ and that $$\nabla^2 F(a_0,b_0) = \begin{pmatrix}
\frac{2}{(1+a_0)^3} & 0\\
0 & \frac{2}{(1-b)^3}
\end{pmatrix},$$ is a diagonal matrix with entries between $10^7$ and $\frac{1}{10^7}$ as long as $[a_0,b_0] \in \left[-\frac{99}{100}, \frac{99}{100}\right]$. If either of $a_0, b_0$ is outside that range, then Lemma  \ref{ifusmallenergybig} gives a contradiction to the assumption on energy. Thus, by the Taylor remainder theorem (and the fact that $F$ is $C^2$ as long as $a$ stays away $-1$ and $b$ stays away from $1$) we have that $$|F(a,b) - F(0,0)| = (a,b)^T\nabla^2 F(a_0, b_0)(a,b) \leq C(a^2 + b^2)$$ for some $(a_0, b_0)$ on the segment connecting $(0,0)$ and $(a,b)$. 

On the other hand, since $v_N(x,-)$ is linear in $y$ and $\hat{w}$ is piecewise linear in $y$ we can see that $$\|v_N - \hat{w}\|_{L^\infty} = \max_{t = a,b} |v_N(x, t) - \hat{w}(t)| =f_N(x)\max_{t = a,b} |t|.$$ Chaining everything together we get that $$\|v_N - \hat{w}\|_{L^\infty} \leq C\sqrt{H_x(\hat{w}) - H_x(v_N(x,-))}.$$

We now estimate $|\hat{w}-w|(t_0)$ for $t_0 \in [-1,1]$. We have two cases; in the first, assume that $t_0 \in [a,b]$. Then $2|\hat{w}(t_0) - w(t_0)| = 2|w(t_0)| \leq \int_a^{t_0}|w_y| + \int_{t_0}^b |w_y|$. Using Jensen's inequality we get that $$|w(t_0)| \leq \left(\frac{b-a}{4}\int_a^b |w_y|^2\, dy\right)^{1/2} \leq C\sqrt{H_x(w) - H_x(\hat{w})}.$$ Putting this together we have that $$|v_N(x, t_0) - w(t_0)| \leq C\sqrt{H_x(w) - H_x(\hat{w})} + C\sqrt{H_x(\hat{w}) - H_x(v_N(x,-))} \leq C\varepsilon^{1/2}.$$

In the second case we assume that $t_0 \in [-1, a]$ (the case that $t_0 \in [b, 1]$ works the same way). Since $w(-1) = \hat{w}(-1)$ and $w(a) = \hat{w}(a)$ we get that $2|w(t_0) - \hat{w}(t_0)| \leq \int_{-1}^a |\partial_y(w- \hat{w})|\, dy$. Again applying Jensen's inequality we get that $$|w(t_0) - \hat{w}(t_0)|  \leq C(a+1) \left(\int_{-1}^a \left|\partial_y w - \frac{f_N}{1+a}\right|^2\, dy\right)^{1/2}.$$ Expanding out the integrand, and using the fact thatat $w(a) - w(-1) = f_N(x)$ we get $$|w(t_0) - \hat{w}(t_0)|  \leq C\left(\int_{-1}^a|\partial_y w|^2\, dy -2\frac{f_N(x)^2}{1+a} + \frac{f_N(x)^2}{1+a}\right)^{1/2} \leq C\left(H_x(w) - H_x(\hat{w})\right)^{1/2}.$$ We can chain the inequalities together as above to get the desired result. 
\end{proof}

From here, we have an easy corollary: outside of the central box, the zero set of $u_N$ is contained in a very thin strip around $\{y =0\}$.

\begin{corollary}\label{c:reifenbergflat}
Let $\delta, \theta > 0$. There exists an $N_0 = N_0(\delta, \theta) > 0$ such that
for $N > N_0$ and every pair $(x,y)\in R_N$ such that 
$|x| < 3N - \delta$, $|y| \leq 1-\frac{1}{88}$, $f_N(x) \geq 1$, and $u_N(x,y) = 0$,
we have $|y| < \theta$.
\end{corollary}

\begin{proof}
Fix $\delta > 0$ and let $N > 0$ be big enough so that, invoking Theorem \ref{t:uniformlip}, 
we can say $u_N$ is $L$-Lipschitz in $\left[-3N + \delta, 3N - \delta\right]\times \left[-1+\frac{1}{44}, 1-\frac{1}{44}\right]$.

Assume there exists a point $(x_0, y_0)$ such that $f_N(x_0) \geq 1$, 
$\theta \leq |y_0| \leq 1-\frac{1}{88}$, and
$u_N(x_0, y_0) = 0$. We know that $|y_0| < 1-\frac{1}{44}$, by 
Lemma \ref{l:nogostripone}.
Then there exists an interval $I$ of length at least $\frac{\theta}{2L}$ such that $|u_N(x, y_0) - y_0f_N(x)| \geq \frac{\theta}{2}$ and $f_N(x) \geq 1$ on all $x\in I$. 

By Lemma \ref{error}, this implies that $H_x(u_N(x,-)) \geq H_x(v_N(x,-)) + \frac{\theta^4}{4C}$ for almost every $x\in I$. Integrating and letting $Q = I\times [-1,1]$ we get that $S_Q(u_N) \geq S_Q(v_N) + C^{-1}\theta^5$ where $C= C(\delta) > 0$ is independent of $N, \theta$. This contradicts \eqref{slicedQ} as long as $N$ is large enough (depending on $C, \theta$ and thus on $\delta, \theta$). 
\end{proof}

\section{The Proof of Theorem \ref{t:main}: Ruling out One-Phase Points}\label{ss:onephaseno}

The main goal of this section is to finish up the proof of our main Theorem \ref{t:main}, 
that $\Gamma_{\mathrm{BP}}(u_N) \neq \emptyset$. In fact, we have the following
more precise description.

\begin{theorem}\label{t:revisedmain}
Let $u_N, R_N, f_N$ be as above. Then, for $N$ large enough, there exists a ``pool of zeroes", 
i.e., a connected open set $\mathcal O \subset \{u_N =0\}\cap \{|x| < 2N+1\}
\cap \{ |y| \leq 1-\frac{1}{44}\}$ such that
$|\mathcal O| > 0$, $\partial {\mathcal O}$ is contained in the free boundary 
$\Gamma^+ \cup \Gamma^-$,
and $\partial {\mathcal O}$ meets $\Gamma^+$,
$\Gamma^-$, and the set, $\Gamma_{\mathrm{TP}}(u)$, of branch points. 
\end{theorem} 

The idea is that Lemma \ref{l:doesntdiplow} guarantees the existence of a ``pool of zeroes" separating the positive and negative phases in the 
 central part of $R_N$. We then want to show that this pool does not ``leak" to the sides of $R_N$.
For this we need the following lemma, whose proof will be the main goal of this section.

\begin{lemma}\label{l:noonephase}
Let 
\[
R_{ext} = \big\{ (x,y) \in R_N \, \mid \, 2N +1< |x| < 3N-1 \text{ and } |y| \leq 1-\frac{1}{88}\big\}.
\]
There exists a $N_0 > 1$ such that if $N \geq N_0$, 
then every free boundary point for $u_N$ in $R_{ext}$
is a two-phase point. Or, put another way: 
$$
R_{ext} \cap (\Gamma^+(u_N)\cup \Gamma^-(u_N))  \subset \Gamma_{\mathrm{TP}}(u).
$$
\end{lemma}

Before we prove the lemma, let us see how its proof implies the theorem. 

\begin{proof}[Proof of Theorem \ref{t:revisedmain} assuming Lemma \ref{l:noonephase}]
By Lemma \ref{l:doesntdiplow}, $u_N$ vanishes in the region where $|x| \leq N-1$ and 
$|y| \leq 1/8$. Denote by ${\mathcal O}_0$ the interior of $\{ u_N = 0 \}$, and let
$\mathcal O$ the connected component of $B(0, 1/8)$ in ${\mathcal O}_0$.

Let us first check that $\partial\mathcal O$ contains one-phase points of both types.
Consider the line segment $\ell$ from the origin to $z_+ = (0,1/44)$; we know that 
$u(z_+) \geq 1/8$ (by \eqref{3a2}), so $\ell$ meets $\partial\mathcal O$. At the first point
of intersection $z=(0,y)$ (going up from $0$), Lemma \ref{l:doesntdiplow} says that $y > \alpha/8$,
and then Lemma \ref{l:doesntdiplow} says that $u(w) \geq 0$ for $w$ near $z$; hence $z$
is a (positive) one-phase point. Similarly, the first point of $\partial\mathcal O$ on the interval
from $0$ to $z_- = (0,-1/44)$ is a negative one-phase point.

Next we want to show that ${\mathcal O} \cap R_{ext} = \emptyset$. Suppose not, and 
let $\gamma$ be a path in ${\mathcal O}$ that goes from $0$ to some point of 
$\mathcal O_0\cap R_{ext}$. Certainly $\gamma$ does not get close to the top and bottom boundaries, i.e.
where $|y| = 1-\frac{1}{88}$, by \eqref{3a2}. So there is a point $(x_0,y_0)\in \gamma$
such that $|x_0| > 2N+1$ and $|y_0| \leq 1-1/44$. Then, as above with the origin, we can find a point 
$P = (x_0,y_0')$ above $(x_0,y_0)$, which lies in $\partial\mathcal O$. This vertical segment lies inside of $\mathcal O$ (and is non-empty by the openness of $\mathcal O_0$). By Lemma \ref{l:noonephase},
this point is a two-phase point. But in fact the proof of Lemma \ref{l:noonephase} will say more: 
near $P$, the free boundary is a Lipschitz graph, with a small constant, and then the non-degeneracy
of $u$ shows that on the vertical line that goes through $P$, $u$ is (strictly) positive on one side of
$P$, and negative on the other side; this contradicts the fact that the open segment between
$(x_0,y_0)$ and $P$ lies in $\mathcal O$. Hence ${\mathcal O} \cap R_{ext} = \emptyset$. Note this, with Lemma \ref{l:nogostripone}, implies that $$\mathcal O \subset R_{in} := \{(x,y) \in R_N \mid |x| \leq 2N+1, |y| \leq 1-\frac{1}{44}\}.$$

We still need to show that ${\mathcal O}$ contains a branch point. Suppose not, and let
$z\in \partial\mathcal O$ be given. Obviously $u$ takes nonzero values near $z$, so $z$
lies in the free boundary, and by assumption $z$ is a one-phase point (since only one phase points or branch points can be on the boundary of an open subset of $\{u =0\}$). Suppose $z \in \Gamma^+(u)$.
In the present situation (and even in ambient dimension $3$), the free boundary in a neighborhood of $z$
is a smooth hypersurface $\Gamma$, with $u> 0$ on one side of $\Gamma$, and $u=0$ on the other
side. Thus $B_r(z) \cap \partial\mathcal O \subset \Gamma^+(u)$ for some $r > 0$ small enough (depending on $z$). 

More globally, the curve $\Gamma \subset \partial\mathcal O$ that contains $z$ is a Jordan curve
(it is disjoint from $R_{ext}$, does not touch the boundary of $R$, and is locally smooth), so $\mathcal O$, which is connected, is contained
in one of the two components of $U$ = $\R^2 \setminus \Gamma$. If $\mathcal O$ is contained in the unbounded component of $U$, then we can replace $u$ with $0$ on the bounded component of $U$ and keep $u$ a valid competitor 
(because $\Gamma \subset \partial \mathcal O \subset R_{in} \subset\subset R_N$). However, this will strictly decrease energy which is a contradiction. Thus $\mathcal O$ is contained in the bounded component of $U$.  Arguing as before, if $\mathcal O$ is not the entirety of this bounded component, then we could replace $u$ by $0$ on the rest of this bounded component and decrease energy. So it must be that $\mathcal O$ is one of the connected components of $\mathbb R^2 \setminus \Gamma$. 

This contradicts the fact that $\partial{\mathcal O}$ meets $\Gamma^-$ too. So $\partial{\mathcal O}$
contains a branch point, and the theorem follows from the proof of the lemma.

In ambient dimension $3$, the same proof would work, using the fact that a connected
smooth orientable hypersurface in $\R^n$ always separates $\R^n$ in exactly two connected component,
as in the Jordan curve theorem.
\end{proof}

The rest of this section will be devoted to the proof of Lemma \ref{l:noonephase}. 
We begin by observing that when $|x| > 2N$, $\partial_y u_N(x,-)$ must be close to 
$\partial_y v_N(x,-) = 2$
at most points:

\begin{lemma}\label{l:closeto2}
For every $\varepsilon > 0$ there exists an $N_0 = N_0(\varepsilon) > 0$ such that if $N \geq N_0$ then 
$$
\int_{\{(x,y) \ | \ 2N < |x| < 3N-\frac{1}{4}, \ |y| \leq 1\}} \left|\frac{\partial u_N}{\partial y} - 2\right|^2\, dxdy \leq \varepsilon.
$$
\end{lemma}

\begin{proof}
Recall that $f_N(x) = 2$ for $2N \leq |x| \leq 3N$, so $v_N(x,y) = 2 y$ and $\partial_y v_N(x,-) = 2$.
Also, $y\mapsto v_N(x,y)$ is harmonic on $[-1,1]$ and (when $x \notin X_0$) 
$y \mapsto u_N(x,y)$ is a $W^{1,2}$ function with the same trace as $v_N$. 
Fix such an $x$, and set $I(w) = \int_{-1}^1 |w_y|^2\, dy$ for $w\in W^{1,2}([-1,1])$. Amongst $w\in W^{1,2}([-1,1])$
with the same boundary values $\pm 2$ as $v_N(x,-)$,
$v_N(x,-)$ minimizes $I$. Thus the Euler-Lagrange equation shows that  $\int_{-1}^1 \partial_y v_N (\partial_y u_N-\partial_y v_N) dy=0$. Hence by Pythagorus
\[
\int_{-1}^1|\partial_y u_N|^2 = \int_{-1}^1 |\partial_y v_N|^2dy 
+ \int_{-1}^1|\partial_y u_N-\partial_y v_N|^2dy.
\]
Now we also care about the functional $H_x$, so we need to add $|\{y \mid w(y) \neq 0\}|$
to $I(w)$. For $v_N(x,-)$, we get $2$, since $v_N(x,y)=2y$ only vanishes at $0$.
For $u_N(x,-)$, Corollary~\ref{c:reifenbergflat} says that for any given $\theta > 0$, and
if $N$ is large enough $u_N(-)$ can only vanish for $|y| < \theta$. Then $2- |\{y \mid u_N(x,y) \neq 0\}| \leq 2\theta$ and 
\begin{align*}
\int_{-1}^1|\partial_y u_N-2|^2dy &= \int_{-1}^1|\partial_y u_N-\partial_y v_N|^2dy
= \int_{-1}^1|\partial_y u_N|^2\, dy - \int_{-1}^1 |\partial_y v_N|^2\, dy 
\\ &= H_x(u_N(x,-)) - H_x(v_N(x,-)) + 2 -|\{y \mid u_N(x,y) \neq 0\}|
\\ &\leq H_x(u_N(x,-)) - H_x(v_N(x,-)) + 2 \theta
\end{align*}
We now integrate over $x$ such that $2N < |x| < 3N-\frac{1}{4}$, use % invoke 
\eqref{m1}, 
and get the desired result for $N\ge N_0(\varepsilon)$. 
\end{proof}

%% Guy: sorry for this, I did not trust all the signs in particular for the second line.
%This, Lemma \ref{slicedmin} plus Corollary \ref{c:reifenbergflat} (with $\theta = \varepsilon$) implies that  
%\begin{align*}
%H_x(u_N(x,-)) &= \int_{-1}^1|\partial_y u_N(x,y)|^2dy + |\{u_N(x,-)\neq 0\}\cap[-1,1]|  \\
%&=\int_{-1}^1|\partial_y u_N-\partial_y v_N|^2dy+\int_{-1}^1|\partial_y v_N|^2dy+2\int_{-1}^1|\partial_y u_N-\partial_y v_N||\partial_y u_N|\\
%& + |\{u_N(x,-)\neq 0\}\cap[-1,1]|\\
%&=H_x(v_N(x,-))+\int_{-1}^1|\partial_y u_N-\partial_y v_N|^2dy+2\int_{-1}^1|\partial_y u_N-\partial_y v_N||\partial_y u_N|\\
%&+ |\{u_N(x,-)\neq 0\}\cap[-1,1]|- |\{v_N(x,-)\neq 0\}\cap[-1,1]|\\
%&\geq H_x(v_N(x,-)) + \int |\partial_y v_N - \partial_y u_N|^2 \\
%&+2-|\{u_N(x,-)= 0\}\cap[-1,1]|- |\{v_N(x,-)\neq 0\}\cap[-1,1]|\\
%&\ge H_x(v_N(x,-)) + \int |\partial_y v_N - \partial_y u_N|^2 -2\varepsilon.
%\end{align*}

We now need to invoke the results of \cite{DeSpVe21} to show that the free boundary is smooth and that the positive and negative parts of $u_N$ extend smoothly to the free boundary. Recall that $u_N^+ = \max\{u_N, 0\}$ and $u_N^- = \max\{-u_N, 0\}$.

\begin{lemma}\label{l:smooth}
Let $\delta > 0$. There exists  $N_0 = N_0(\delta) > 1$ such that if $N \geq N_0$, then each $\partial \{u^{\pm}_N  >0\} \cap (\{ 2N< |x| < 3N-1\}\times[-1,1])$ is a $C^{1,1/4}$-graph over the set $\{y = 0\}$ with norm $\leq \delta$. Furthermore, 
\[|\nabla u_N^{\pm}| \in C^{0,1/4}\left(\overline{\{u_N^{\pm} > 0\}}\cap (\{2N \leq |x| \leq 3N-1-\delta\}\times[-1,1])\right),\]
with a $C^{0,1/4}$ seminorm less than $1$.
\end{lemma}

\begin{proof}
Pick some $r_0 > 0$ small, but independent of $N, \delta$. For points $x_0$ such that $B(x_0, r_0) \cap \Gamma_{\mathrm{TP}}(u_N) = \emptyset$, the result follows from uniform Reifenberg flatness (Corollary \ref{c:reifenbergflat}) and standard ``flat-implies-smooth" regularity for the one-phase problem (c.f. \cite{AlCa81}). 

 If $B(x_0, r_0) \cap \Gamma_{\mathrm{TP}}(u_N) \neq \emptyset$, then we can consider $y\in B(x_0, r_0) \cap \Gamma_{\mathrm{TP}}(u_N)$ and look at $B(y, 2r_0)$. Then the regularity follows from the $\theta$-Reifenberg flatness at scale 1 (Corollary \ref{c:reifenbergflat}), the uniform Lipschitz continuity (Theorem \ref{t:uniformlip}) and \cite[Theorem 3.1]{DeSpVe21}. 

In both instances, the dependence on $r_0, \theta$ is such that the $C^{1,1/4}$-norm of the graph(s) goes to zero as $r_0 > 0$ stays constant but the Reifenberg flatness parameter, $\theta$, goes to zero. Corollary \ref{c:reifenbergflat} says we can take $\theta$ arbitrarily small, at the expends of making $N$ larger. The regularity of the gradient is a consequence of standard elliptic regularity once  we know the regularity of the free boundary. 
\end{proof}

We are finally ready to prove Lemma \ref{l:noonephase}:

\begin{proof}
Let $\varepsilon > 0$ be small, to be chosen later and assume by contradition that there exists a one phase point $(x_0, y_0) \in \partial\{u_N > 0\}\cap \{2N\leq |x| \leq 3N-1\}$. By Lemma \ref{l:closeto2} we may choose $N$ large enough (depending only on $\varepsilon$) such that for any square, $Q$, of side length $\ell(Q) \leq 3/4$ centered at the point $(x_0,y_0)$ we have 
\begin{equation}\label{4a1}
\int_Q \left|\frac{\partial u_N}{\partial y}- 2\right|^2\, dx dy \leq \varepsilon.
\end{equation}
% $$\int_Q \left|\frac{\partial u_N}{\partial y}- 2\right|^2\, dx dy \leq \varepsilon.$$ 
Using Fubini we also get that $$\int_0^{3/4}\int_{\partial Q}\left|\frac{\partial u_N}{\partial y} - 2\right|^2 \, d\mathcal H^{1} d\ell(Q) \leq \varepsilon,$$ where the integration is occurring over squares all centered at $(x_0, y_0)$ with increasing side lengths. Thus we can pick a cube $Q_0$  (which will be fixed going forward) with $1/2 \leq \ell(Q_0) \leq 3/4$ such that \begin{equation}\label{e:smallinint}\int_{Q_0} \left|\frac{\partial u_N}{\partial y}- 2\right|^2\, dx dy, \int_{\partial Q_0} \left|\frac{\partial u_N}{\partial y}- 2\right|^2\, d\mathcal H^1 \leq 4\varepsilon.\end{equation} Note that $u_N|_{Q_0}$ is $L$-Lipschitz (by Theorem \ref{t:uniformlip}), where $L$ is independent of $N, Q_0$.

%Consider a sequence of squares, centered at $(x_0, y_0)$, of increasing size. By Fubini and Lemma \ref{l:closeto2}, we can chose $N$ large enough (depending only on $\varepsilon$) to guarantee the existence of a square $Q$ of side length between $1/2$ and $3/4$ such that \begin{equation}\label{e:smallinint}\int_{Q} |\partial_y u_N - 2|^2\, dx dy < \varepsilon\end{equation} and $\int_{\partial Q}|\partial_y u_N - 2|^2\, ds < \varepsilon$. 

Let $\tilde{Q} = Q_0\cap \{u_N > 0\}$. By Lemma \ref{l:smooth}, the domain $\tilde{Q}$ is piece-wise $C^{1,1/4}$ and is an NTA domain, with constants uniform in $N$ (c.f. \cite{JeKe82} for definitions and details). Let us say a bit more about this; the NTA constants depend on how the vertical edges of $Q$ touch the smooth graph $\partial \{u_N > 0\}$. However, this graph over $\{y = 0\}$ has very small norm (uniform in $N$) so this intersection happens (quantitatively) transversely and thus the NTA constants are also uniform in $N$. These bounds on norm of the graph which gives $\partial \{u_N > 0\}$ also imply that 
$4 \geq \mathcal H^1(\tilde{\partial Q}) \geq 1$ 
and $|\tilde{Q}| \geq \frac{1}{8}$. From this information, using \eqref{e:smallinint} and Chebyshev, there exists a $A \in \tilde{Q}$ with $\mathrm{dist}(A, \partial \tilde{Q}) > \frac{1}{100}$ and $|\partial_y u(A) - 2|^2 < 8\varepsilon$ (this will work as long as $\varepsilon > 0$ is small enough).

 Let $\omega_N^A$ be the harmonic measure of $\tilde{Q}$ with a pole at $A$ (where the notation emphasizes the dependence on $N$). Since $\partial_y u$ is a harmonic function in $\tilde{Q}$ we have the following integral representation: $$\partial_y u_N(A) = \int_{\partial \tilde{Q}} \partial_y u_N(P)d\omega_N^A(P).$$
 %By the argument above, and standard results in the harmonic measure literature (e.g. \cite[Theorem 2]{DaJe90}), we have that $\omega_N^A\in A_\infty(\mathcal H^1|_{\partial \tilde{Q}})$, with constants that are uniform in $N$. In particular, for every $\theta \in (0,1)$ there exists $\eta > 0$ such that for every $F\subset \partial \tilde{Q}$ \begin{equation}\label{e:Ainfinity} \frac{\mathcal H^1(F)}{\mathcal H^1(\partial \tilde{Q})} < \eta \Rightarrow \omega^A_N(F) < \theta.\end{equation}
 
 Finally, it will be convenient to let $\partial \tilde{Q} = \Gamma_1 \cup \Gamma_2$ where $\Gamma_1 = \partial \{u_N > 0\} \cap \overline{Q}_0$ and $\Gamma_2 = \partial Q_0 \cap \overline{\{u_N > 0\}}$, see the picture below.
 
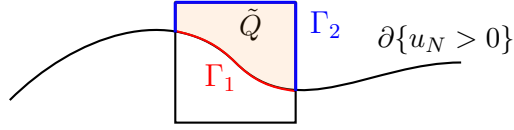
\vspace{0.5cm}
 \begin{center}
 \begin{tikzpicture}[scale=1.0]
\draw [thick, smooth] (-3,-0.5) to[out=45, in=135] (0,0) to[out=315, in=180] (3,0);
\draw[name path = A] [thick, red, smooth] (-0.8,0.42) to[out=347, in=135] (0,0) to[out=315, in=173] (0.8,-0.38);
\draw [thick, smooth] (-0.8,-0.8) rectangle (0.8, 0.8);
\draw [very thick, blue, smooth] (-0.8, 0.8) -- (-0.8, 0.4);
\draw [very thick, blue, smooth] (0.8, 0.8) -- (0.8, -0.38);
\draw[name path = B] [very thick, blue, smooth] (-0.8,0.8) -- (0.8,0.8);
\node at (1.2, 0.5) {\textcolor{blue}{$\Gamma_2$}};
\node at (-0.2, -0.2) {\textcolor{red}{$\Gamma_1$}};
\node at (2.8,0.3) {\small $\partial\{u_N>0\}$};
\node at (0.2, 0.5) {\small $\tilde{Q}$};
\tikzfillbetween[of=A and B]{orange, opacity=0.1};
\end{tikzpicture}
\end{center}

Recall that $(x_0, y_0)$ is a one-phase point for $u_N$. Thus we can compute $|\partial_y u_N(x_0, y_0)| \leq |\partial_\nu u_N(x_0, y_0)| = 1$, where the last equality holds due to the free boundary condition at (regular) one-phase points. Because the derivative of $u_N$ restricted to $\partial \{u_N > 0\}$ has $C^{0, 1/4}$-semi-norm less than 1, we also have that 
$$
\begin{aligned} |\partial_y u_N(x_0,y_0) - \partial_y u_N(P)| <& 
\left(\frac{3}{4}\right)^{1/4}
< \frac{94}{100} \ \text{ for all } P \in \Gamma_1 
\\
\Rightarrow \partial_y u_N(P)\leq&  2-\frac{1}{20} % ,\,\, 
\ \ \text{ for all } % \forall 
P \in \Gamma_1.
\end{aligned}
$$ 
Putting things together, we have that \begin{equation}\label{e:reducetoavg}\begin{aligned}
2-3\varepsilon^{1/2} \leq \partial_y u_N(A) \leq& \left(2-\frac{1}{20}\right)\omega_N^A(\Gamma_1) + \int_{\Gamma_2} \partial_y u_N(P)d\omega_N^A(P)\\
\Rightarrow 2(1-\omega_N^A(\Gamma_1)) + \frac{\omega_N^A(\Gamma_1)}{20} \leq& \int_{\Gamma_2} \partial_y u_N(P)d\omega_N^A(P) + 3\varepsilon^{1/2}\\
\Rightarrow 2 + c \leq& \fint_{\Gamma_2}\partial_y u_N(P)d\omega_N^A(P),
\end{aligned}\end{equation}
where in the last line we used that $\omega_N^A(\Gamma_2) = 1-\omega_N^A(\Gamma_1)$ and also Bourgain's Lemma (c.f. \cite[Lemma 4.2]{JeKe82}), which implies that there is a constant $\tilde{c} > 0$ (independent of $\varepsilon, N$) such that $\frac{\omega_N^A(\Gamma_1)}{1-\omega_N^A(\Gamma_1)} \geq \tilde{c}$ and $\tilde{c}^{-1} \geq \frac{1}{1-\omega_N^A(\Gamma_1)}$. 

Recall that $u$ is $L-$Lipschitz and write 
$\Gamma_2 = \Gamma_{2,+} \cup \Gamma_{2,-}$,
with 
\[
\Gamma_{2,+} = \{P\in \Gamma_2 \mid L \geq |\partial_y u_N(P)| > 2 + c/2\}
\ \text{ and } \Gamma_{2,-} %= \Gamma_2 \setminus \Gamma_{2,+}
= \{P \in \Gamma_2 \mid |\partial_y u_N(P)| \leq 2 + c/2\}
\]
After overestimating $\fint_{\Gamma_2} \partial_y u_N d\omega_N^A$ over each set, \eqref{e:reducetoavg}, gives us 
\begin{equation}\label{e:lowerboundratio} 
\frac{c}{2} \leq L \, 
\frac{\omega_N^A(\Gamma_{2,+})}{\omega_N^A(\Gamma_2)}.
\end{equation} 
Once more invoking Bourgain's theorem we have that
$$\theta: = \frac{c\omega_N^A(\Gamma_2)}{2L}$$ 
is bounded strictly away from zero, independently of $N$ (large enough) or $\varepsilon > 0$. 

Using the condition on the integral of $|\partial_y u_N - 2|$ on $\partial Q_0$ 
(i.e. \eqref{e:smallinint}) we see that 
\[
\mathcal H^1(\Gamma_{2,+})
< \frac{16\varepsilon}{c^2}.
\] 
Recall that $\mathcal H^1(\partial \tilde{Q}) \geq 1$ and we get that 
\begin{equation}\label{e:testainfinity}
\frac{\mathcal H^1(\Gamma_{2,+})}{\mathcal H^1(\partial \tilde{Q})}
< \frac{16\varepsilon}{c^2} \ll \theta 
\equiv  \frac{\omega_N^A(\Gamma_{2,+})}{\omega_N^A(\partial \tilde{Q})}.
\end{equation}

We now recall that in $\tilde{Q}$ we have that the harmonic measure $\omega_N^A \in A_\infty(\mathcal H^1)$ (see e.g. \cite[Theorem 2]{DaJe90}). In fact, the $A_\infty$-constants depend on the NTA constants of $\tilde{Q}$, the $1$-Ahlfors regularity of $\partial \tilde{Q}$, the distance from $A$ to $\partial \tilde{Q}$ and the diameter of $\tilde{Q}$ (for more details and definitions of the relevant terms, see \cite{DaJe90}). As discussed above, all of these quantities can be taken uniform for $N$ large enough. Thus we can take $\varepsilon > 0$ small until \eqref{e:testainfinity} contradicts $\omega_N^A \in A_\infty(\mathcal H^1)$ and we are done. 
\end{proof}

\begin{remark}\label{r:higherdim}
The arguments above can be adapted to produce
cusp points in $\mathbb R^{2+1}$,
but not directly in ambient dimensions larger than $3$. In the setting of $\R^3$,
our domain is given by $R_N := B'(0, 3N)\times [-1,1] \subset \mathbb R^{2+1}$ 
where $B'$ is a ball in $\mathbb R^{2}$. 
Then $f_N(r, \theta) = f_N(r): B'(0, 3N)\rightarrow \mathbb R$ depends only on the radial variable. Now we define $f_N$ piecewise: 

\begin{align*}
f_N(r)=\begin{cases} 1-\alpha, &\text{ if } r\le 1\\
\left(\frac{\log(r)}{\log(2N)} + 1\right)(1+\alpha) -2\alpha &\text{ if } 1\le |x|\le 2N\\
2, &\text{ if } 2N\le r\le 3N.
\end{cases}
\end{align*}

We can then define the slice minimizer similarly, where 
$v_N(r, \theta, -) \in W^{1,2}([-1,1])$ minimizes with the boundary values 
$v_N(r, \theta, \pm 1) = \pm f_N(r)$. In particular when $f_N(r) \geq 1$ we have 
$v_N(r, \theta, y) = yf_N(r)$ and when $f_N(r) \leq 1$ we have 
$v_N(r, \theta, y)= \mathrm{sgn}(y)\left(|y|-1 + f_N(r)\right)_+$. 
In either setting we have $|\partial_r v_N| \leq |\partial_r f_N|$. 

Computing just like in Lemma \ref{l:dvdxnotbig}, we get that \begin{equation}\label{e:higherdimnotbig}
\iint_{\overline{R}_N} \left|\frac{\partial v_N}{\partial r}\right|^2 \simeq \int_{1}^{2N} \frac{1}{r^2\log(2N)^2}r\, dr = \frac{1}{\log(2N)} \stackrel{N\rightarrow \infty}{\rightarrow} 0.
\end{equation}

From here we argue identically as above, noting that we never use the precise bound on  $\iint_{\overline{R}_N} \left|\frac{\partial v_N}{\partial r}\right|^2$, just that it goes to zero as $N \rightarrow \infty$, and every other quantity stays bounded. 

For example, the contradiction in the proof of Lemma \ref{l:doesntdiplow} now comes 
when
$\eta \leq \frac{C}{\eta \log(2N)}$, which is not true for $N > 1$ large enough. 
\end{remark}

\section{Accumulating Cusps for Almost Minimizers}\label{s:AMcusps}

In this section we prove that the cusp set for almost-minimizers to \eqref{e:ACF} can be essentially arbitrary. To state our results in maximum generality we introduce the variable coefficient version of \eqref{e:ACF}:

\begin{equation}\label{e:vACF}
    J_\Omega(u) := \int_{\Omega} |\nabla u|^2 + q_+^2(x)\chi_{\{u > 0\}}(x) + q_{-}^2(x)\chi_{\{u < 0\}}\, dx.
\end{equation}

Throughout this section we will assume that $q_{\pm} \in C^{0,\alpha}(\overline{\Omega})$ for some $\alpha \in (0,1)$ and that the weights satisfy the non-degeneracy condition; $q_{\pm} \geq c_0 > 0$ in all of the domain. Clearly we recover the original functional \eqref{e:ACF} 
by letting $q_{\pm} \equiv \lambda_{\pm}$ in \eqref{e:vACF}. 

We now state our main result:

\begin{theorem}\label{t:fbam}
Let $f^- \leq f^+ \in C^2(\mathbb R^n)$ such that $f^+ = f^- = 0$
outside of $B(0, R/10)$ for some large $R> 0$. Let $\Gamma^\pm$ be the graphs of $f^{\pm}$. For any $q_{\pm} \in C^{0,\alpha}$, with $c_0^{-1} \geq q_{\pm} \geq c_0 > 0$, there exists an almost-minimizer $u$ to the energy $J_{B(0, R)}$ such that $\Gamma^\pm = \partial \{\pm u > 0\}$.
\end{theorem}

Notice that for $u$ as in Theorem \ref{t:fbam} we have $\Gamma_{\mathrm{BP}}(u) = \partial_{\mathbb R^{n-1}}\{f^-(x) = f^+(x)\}$. Recall that any closed set can be the zero set of a $C^2$ function (take a smoothing of the distance function to the given set, c.f. \cite[VI, Theorem 2]{Steinbook}). As such, we have the following corollary (compare to Theorem \ref{t:ambranchpoints} from the introduction): 

\begin{corollary}\label{c:anyset}
Let $E \subset \mathbb R^{n-1}$ be any compact set  with no interior point in $R^{n-1}$
and let $q_{\pm} \in C^{0,\alpha}(\mathbb R^n)$ be non-degenerate and bounded. 
Then, for some $R > 0$ large enough depending on $E$, there exists an almost-minimizer $u$ 
to $J_{B(0, R)}$ with weights $q_{\pm}$, such that $\Gamma_{\mathrm{BP}}(u) = E$. 

Furthermore, we can take 
$\Gamma^+ := \partial\{ u > 0\}$
to be the reflection of $\Gamma^- :=  \partial\{ u < 0\}$
around $\{x_n = 0\} \subset \mathbb R^n$. 
\end{corollary}

Theorem \ref{t:fbam} will follow from two lemmas, the first a general result about what functions are almost-minimizers to the two-phase functional. The second, a construction of such functions. 

\begin{lemma}\label{l:fbam}
Let, $n\geq 2, \alpha \in (0,1), f^- \leq f^+ \in C^2(\mathbb R^n)$ 
such that $f^- = f^+ =0$
outside of $B(0, R/10)$ for some $R > 1$,
 and let $\Gamma^{\pm}$ be the graph
 of $f^{\pm}$. 
 Let  $\Omega^{+}$ (resp. $\Omega^-$) be the part of  $B(0, 4R)$
that lies above (resp. below) the graph of 
 $f^{+}$ (resp. $f^-$)
 for some $R > 0$ large. Let $u^{\pm} \in C^{1,\alpha}(\overline{\Omega}^{\pm})$ be such that  there exists a constant $C_1 > 0$ such that for $x \in \Omega^{\pm} \cap B(0,2R)$ we have that

 \begin{equation}\label{e:rdist}
 \begin{aligned} C_1^{-1}\mathrm{dist}(x, \Gamma^{\pm}) 
 \leq u^{\pm} \leq& C_1\mathrm{dist}(x, \Gamma^{\pm})\\ 
|\nabla u^{\pm}| \leq& C_2\\
\|D^2 u^{\pm}\| \leq& C_1\mathrm{dist}(x, \Gamma^{\pm})^{-1}. 
\end{aligned}
\end{equation}
Then $u = u^+ - u^-$ is an almost-minimizer to \eqref{e:vACF} inside of $B(0, R)$ where $q_{\pm}$ are the $C^{0,\alpha}$ functions which agree with $|\nabla u^{\pm}|$ on $\Gamma^{\pm}$.

More precisely, there exists a constant $C = C(C_1, \|f^\pm\|_{C^2}, \|u^\pm\|_{C^{1,\alpha}(\overline{\Omega^{\pm}})}) > 0$ and $1 > r_0 = r_0(C_1, \|f^\pm\|_{C^2}, \|u^\pm\|_{C^{1,\alpha}(\overline{\Omega^{\pm}})})  > 0$ such that for any ball $B$ satisfying $\overline{B} \subset B(0, R)$, and $r(B) \leq r_0$ we have \begin{equation}\label{e:am} J_{B}(u) \leq J_B(v) + Cr^{n+\frac{\alpha}{4n+2\alpha}}\end{equation} for any $v = u$ on $B(0, R)\backslash B$.
\end{lemma}

Key to the proof of Lemma \ref{l:fbam} is the following result which is adapted from \cite{DeJe09}:

\begin{lemma}\label{l:slidingbarriers}
Let $v$ be a critical point to $J_B$ (associated to $q_{\pm}$). Assume there exists, parameterized by $t\in[a,b]$, a family of $\phi_t: \overline{B} \rightarrow \R$ (continuous in both variables) that satisfy the following properties:

\begin{enumerate}
\item $\Delta \phi_t = 0$ in $\{\phi_t \neq 0\} \cap B$.
\item $\{\phi_t = 0\}  = \partial \{\phi_t > 0\} = \partial \{\phi_t < 0\}$.  Furthermore $t\mapsto \{\phi_t = 0\}$ is continuous in the Hausdorff distance sense. 
\item At every point on $\partial \{\pm \phi_t > 0\}$ there exists a ball inside $\{\pm \phi_t > 0\}$ which touches the free boundary at that point. 
\item At every $x_1 \in \partial \{\pm \phi_t > 0\}$ we satisfy $$(\partial_{\nu^+} \phi_t)^2(x_1) - (\partial_{\nu^-} \phi_t)^2(x_1) \geq q^2_+(x_1) - q^2_-(x_1)$$ (respectively $\leq q^2_+ - q^2_-$ for supersolutions)
\item $\phi_b \leq v$ in $\overline{B}$ (respectively $\phi_b\geq v$ for supersolutions).
\item For all $\rho \in [a,b]$, $\phi_\rho \leq v$ on $\partial B$ and $\phi_\rho < v$ on $\partial B \cap \overline{\{v \neq 0\}}$ (with the inequalities reversed for supersolutions)
\end{enumerate}
then we have that $\phi_a \leq v$ in $\overline{B}$ (respectively $\phi_a \geq v$ in $\overline{B}$ for super solutions). 
\end{lemma}

\begin{proof}[Proof of Lemma \ref{l:fbam}, assuming Lemma \ref{l:slidingbarriers}]

We first check the almost-minimization condition \eqref{e:am} for  $x_0 \in \Gamma^+ \cap \Gamma^-$:

\medskip

\noindent {\bf Case 1:} Let $x_0  \in \Gamma^+\cap \Gamma^-$ and $r > 0$ small enough and let $v =u$ on $B(0, R)\backslash B(x_0,r)$. 

Since $u^{\pm} \in C^{1,\alpha}(\overline{\Omega^{\pm}})$ and $\Gamma^+\cap\Gamma^-$ is closed and smooth, there exists an $r_0$ such that if $r < r_0$ then, for $x\in B(x_0, r)$ \begin{equation}\label{e:blowuptp}\begin{aligned} q_+(x_0)((x-x_0)\cdot e)^+ -& q_-(x_0)((x-x_0)\cdot e)^- -r^{1+7\alpha/8} \leq u(x)\\ u(x)\leq q_+(x_0)((x-x_0)\cdot e)^+ -& q_-(x_0)((x-x_0)\cdot e)^- + r^{1+7\alpha/8}.\end{aligned}\end{equation} Recall that $q_\pm(x_0) = |\nabla u^\pm|(x_0)$ by definition.  

Recall the 
functional proved to be monotone by Weiss \cite{Wei99}: $$W(u, x_0, r) \equiv \frac{1}{r^n} \int_{B(x_0, r)} |\nabla u|^2 + q^2_+(x) \chi_{\{u > 0\}} + q_-^2(x) \chi_{\{u < 0\}}\, dx - \frac{1}{r^{n+1}} \int_{\partial B(x_0, r)} u^2\, d\sigma.$$ It follows from \eqref{e:blowuptp} and the $C^{0,\alpha}$-character of $|\nabla u^{\pm}|$ that \begin{equation}\label{e:Wunottoobig} W(u, x_0, r) \leq W(u, x_0, 0) + Cr^{7\alpha/8} = \frac{1}{2}\mathrm{Vol}(B(0,1))(q^2_+(x_0) + q^2_-(x_0)) + r^{7\alpha/8},\end{equation} where $C > 0$ depends on the $C^{0,\alpha}$-norm of $|\nabla u^\pm|$ restricted to $\overline{\Omega}^{\pm}$. 

We now want to show that for any minimizer $v$ to $J_{B(x_0, r)}$ with $v = u$ in $\mathbb R^n \backslash B(x_0, r)$ we have 
\begin{equation}\label{e:squeezed} 
(\partial \{v > 0\}\cup \partial \{v < 0\})\cap B(x_0, r) \subset \{x\in B(x_0, r)\mid |(x- x_0) \cdot e| < r^{1+\alpha/2}\}
\end{equation} 

Assume that \eqref{e:squeezed} holds. We would like to compare $W(v, x_0, r)$ to $W(u, x_0, r)$ but na\"ively underestimating $W(v, x_0, r)$ by $W(v, x_0, 0)$ is problematic as $x_0$ may not be in the free boundary of $v$. To combat this, let $x^+$ be the closest point to 
$x_0$ in $\partial \{v > 0\}$ and $x^-$ the closest point to $x_0$ in $\partial \{v < 0\}$. 
We note that \eqref{e:squeezed} implies that $|x_0 - x^{\pm}| < r^{1+\alpha/2}$. 
Let $\rho = r/2 - \max\{ |x_+-x_0|,  |x_- -x_0|\}$.
Note 
that $B(x^{\pm}, 2\rho) \subset B(x_0, r)$ and
$r > 2\rho > r(1-r^{\alpha/2})$. Also,  $B(x_0, r/2) \supset B(x^{\pm}, \rho)$.
Hence

\begin{equation}
\label{e:Wvnottoosmall1} 
 W(v, x_0, r/2)\geq  W(v^+, x^+, \rho) + W(v^-, x^-, \rho)- Cr^{\alpha/2},
 \end{equation}
because
%In the above inequality we used that $B(x_0, r/2) \supset B(x^{\pm}, \rho)$ that 
$|2\rho/r - 1| < r^{\alpha/2}$ and % that 
$v$ is Lipschitz in $B(x_0, r/2)$ with a constant controlled by % the 
$\frac{1}{r^n}\int_{B(x_0, r)} |\nabla v|^2  \leq \frac{1}{r^n} J_{B(x_0, r)}(u)$, % added , 
the latter of which is bounded by the Lipschitz norm of $u$ and the supremums of $q_{\pm}$. 

To estimate each term in the summand on \eqref{e:Wvnottoosmall1}, we think of $v^+, v^-$ is being separate critical points to the \emph{one-phase problems} associated to weights $\tilde{q}_{\pm}|_{\partial \{v^{\pm} > 0\}} := \partial_{\nu^{\pm}} v^\pm|_{\partial \{v^{\pm} > 0\}}$. By the regularity theory of the two-phase problem in \cite{DeSpVe21} and \eqref{e:squeezed} we know that $\partial \{\pm v > 0\}$ are $C^{1,\alpha}$ in $B(x_0, r/2)$. Thus $\tilde{q}_{\pm} \in C^{0,\alpha}(\partial \{v^{\pm} > 0\})$ and can be extended H\"older continuously to all of $B(x_0, r/2)$ (with norm uniform in the constants we care about). 

 The free boundary condition satisfied by $v$ being a minimizer to the two-phase problem tells us that $\tilde{q}_{\pm}:= \partial_{\nu^{\pm}} v^{\pm} = q_{\pm}$ at one-phase points but at two-phase points the free boundary condition only implies that $\tilde{q}_{\pm}:= \partial_{\nu^{\pm}} v^{\pm} \geq q_{\pm}$. These observations tell us that $\|\tilde{q}_{\pm} - q_{\pm}\|_{L^\infty(B(x_0, r/2))} \leq Cr^{\alpha}$.

 By monotonicity, we can underestimate each $$\begin{aligned}W(v^{\pm}, x^{\pm}, \rho) \geq& \frac{1}{\rho^n}\int_{B(x^\pm, \rho)} |\nabla v^{\pm}|^2 + \tilde{q}^2_\pm(x)\chi_{\{\pm v > 0\}}\, dx - \frac{1}{\rho^{n+1}}\int_{\partial B(x^{\pm}, \rho)} (v^{\pm})^2 \, d\sigma - Cr^\alpha\\ \geq& \frac{1}{2}\mathrm{Vol}(B(0,1))\tilde{q}^2_{\pm}(x^{\pm}) -Cr^{\alpha} \\ \geq& \frac{1}{2}\mathrm{Vol}(B(0,1))q^2_{\pm}(x^{\pm}) - Cr^{\alpha} \\ \geq& \frac{1}{2}\mathrm{Vol}(B(0,1))q_{\pm}^2(x_0) - Cr^{\alpha}.\end{aligned}$$

Putting this back into \eqref{e:Wvnottoosmall1}, using monotonicity 
and \eqref{e:Wunottoobig} we get that 
\begin{align*}
W(v, x_0, r) & \geq 
W(v, x_0, r/2) - Cr^{\alpha/2} 
\geq \frac{1}{2}\mathrm{Vol}(B(0,1))(q_+^2(x_0) + q_-^2(x_0)) 
- Cr^{\alpha/2}
\\
& \geq W(u, x_0, r) - Cr^{\alpha/2}.
\end{align*} Multiplying through by $r^n$ gives the almost-minimization inequality. 

So to finish Case 1, it suffices to prove \eqref{e:squeezed}. Here is where we apply Lemma \ref{l:slidingbarriers}. We do this on the interval $[a,b] = [r^{1+\alpha/2}, 3r]$ (recall that we can take $r$ small so that $3r < 1$). To simplify notation, let us assume that $x_0 = 0$ and $e = e_n$. We then create the family of sub/super-solutions to the two phase problem in $B(x_0,r)$ defined by $$\overline{w}_t = M_+(x_n -t)^+ - m_-(x_n-t)^-, \qquad \underline{w}_t = m_+(x_n + t)^+ - M_-(x_n +t)^-,$$ where $m_\pm = \min_{B(x_0, r)} q_\pm$ and $M_\pm = \max_{B(x_0, r)} q_{\pm}$. Let $v$ be a minimizer to the two phase problem in $B(x_0, r)$ with boundary values $u$.

We first verify condition (5) in Lemma \ref{l:slidingbarriers}. We note that for any $x\in B(x_0, r)$ $$\overline{w}_{3r}(x) \leq -2m_-r \leq -2q_-(x_0)r +Cr^{1+\alpha} \stackrel{\eqref{e:blowuptp}}{\leq} \inf_{\partial B(x_0, r)} u = \inf_{\partial B(x_0,r)} v \leq v(x)$$ where the last inequality follows from the minimum principle applied to the superharmonic function $-v^-$ and the middle inequalities follow from the $C^{0,\alpha}$ continuity of $q_-$ and the fact that we can take $r_0 \ll \inf_{\overline{B(0, R)}} q_-$. 

 To check condition (6) in Lemma \ref{l:slidingbarriers} we will show that 
 $\overline{w}_t|_{\partial B(x_0, r)} < u|_{\partial B(x_0, r)}$ 
 for all $r^{1+7\alpha/8}< r^{1+\alpha/2} \leq t \leq 3r$. When $x_n > t$ we have 

 $$\begin{aligned} \overline{w}_t(x) = M_+(x_n -t) \leq& q_+(x_0)(x_n - t) + (\mathrm{osc} q_+)(|x_n| + |t|)\\ \leq& q_+(x_0) x_n -q_+(x_0)r^{1+\alpha/2} + Cr^{1+\alpha} < q_+(x_0)x_n - r^{1+7\alpha/8} \stackrel{\eqref{e:blowuptp}}{\leq} u(x),\end{aligned}$$ again assuming $r_0$ is small enough. 
 
 When $0 \leq x_n < t$ we note that $\overline{w}_t(x) < 0$. So there is only something to prove if $u(x) < 0\stackrel{\eqref{e:blowuptp}}{\Rightarrow} x_n < Cr^{1+7\alpha/8}$. In this case $x_n < t/2$ (for $r$ small enough) and we have $$\overline{w}_t(x) \leq m_-(x_n -t)  \leq - \frac{m_-}{2}t \leq -\frac{m_-}{2}r^{1+\alpha/2} <  -r^{1+7\alpha/8} \stackrel{\eqref{e:blowuptp}}{\leq} u(x).$$
 
 When $x_n < 0$ we have $$\overline{w}_t(x) = m_-(x_n -t) \leq q_-(x_0)(x_n -t) + Cr^{1+\alpha} \leq q_-(x_0) x_n - Cr^{1+\alpha/2} \stackrel{\eqref{e:blowuptp}}{<} u(x).$$
 
 Condition (1) in Lemma \ref{l:slidingbarriers} follows from the fact that $\overline{w}_t$ is linear where it is not zero. Conditions (2) and (3) follow from the observation that $\{\overline{w}_t = 0\} = \{x_n = t\}$ which moves continuously with $t$. Finally the free boundary condition, Condition (4), follows from the definition of $M_{\pm}, m_{\pm}$. 
 
Thus by Lemma \ref{l:slidingbarriers} above (which
will be checked below)
it must be the case that $\overline{w}_{r^{1+\alpha/2}} \leq v$ in all of $B(x_0, r)$. A similar argument shows that $\underline{w}_{r^{1+\alpha/2}} \geq v$ in all of $B(x_0, r)$. Thus $\partial\{v > 0\}, \partial \{v < 0\} \subset \{x\in B(x_0, r)\mid |x_n| \leq r^{1+\alpha/2}\}$.

\medskip

\noindent {\bf Case 2:} Now we assume that $x_0 \in \Gamma^+ \cup \Gamma^-$. Let $M > 1$ to be chosen later. If $R$ is such that $B(x_0, Mr) \cap (\Gamma^+\cap \Gamma^-) \neq \emptyset$ then note that $B(x_0, r) \subset B(y, 2Mr)$, for some $y\in \Gamma^+\cap \Gamma^-$. As such, if $u = v$ in $B(0, R) \backslash B(x_0, r)$ then $u = v$ in $B(0, R) \backslash B(y, 2Mr)$. So Case 1 tell us that $u$ satisfies the almost-minimization criterion in these balls with some $\tilde{C} = C(2M)^{n+\alpha/2}$ (where $C > 0$ comes from Case 1). 

If $B(x_0, Mr) \cap (\Gamma^+\cap \Gamma^-) = \emptyset$ we assume, without loss of generality, that $x_0 \in \Gamma^+ \backslash \Gamma^-$ and we assume that $B(x_0, r)\cap \Gamma^- \neq \emptyset$ (otherwise $u^-$ is identically zero 
on $B(x_0,r)$
and the analysis is even easier). 
Pick
$y\in B(x_0, r)\cap \Gamma^-$. As long as $M > 3$ then $B(y, 2r) \cap (\Gamma^+\cap \Gamma^-)  = \emptyset$ and $u^+, u^-$ %% with 
satisfy one-phase versions of \eqref{e:blowuptp} in $B(x_0, 2r)$ 
and $B(y, 2r)$ respectively. To show that $u^+$ and $u^-$ each satisfy the correct energy inequalities for the one-phase versions of $J$ inside of $B(x_0, r)$ and $B(y, r)$ we proceed as in Case 1 (but the argument is simpler because we need only care about one-phase functionals). The almost-minimization property for $u$ in $B(x_0, r)$ for the two-phase functional then follows since the two-phase energy is the sum of the one-phase energies. We omit these details. 

\medskip 

\noindent {\bf Case 3:} Assume that $x_0 \notin \Gamma^+ \cup \Gamma^-$.  Let $M > 1$ be as above (perhaps enlarge a bit, to be determined later) and let $\varepsilon \ll 1$ (to be determined later). If $B(x_0, Mr^{1-\varepsilon}) \cap (\Gamma^+ \cup \Gamma^-) \neq \emptyset$ then the almost minimization criterion in $B(y, 2Mr^{1-\varepsilon})$\footnote{It may be surprising that we break scaling by considering a radius of size $r^{1-\varepsilon}$ but we are merely taking advantage of the lack of scale invariance in the characterization of almost-minimizers} for some $y\in \Gamma^+\cup\Gamma^-$ shows that $$J_{B(x_0, r)}(u) - J_{B(x_0, r)}(v)  \leq C(2Mr^{1-\varepsilon})^{n+\alpha/2} = C_{M} r^{(1-\varepsilon)(n+\alpha/2)} \leq C_M r^{n+\alpha/4},$$ for any $u = v$ outside of $B(x_0, r)$. Note this last inequality above holds for $r< 1$ as long as $(1-\varepsilon)(n+\alpha/2) \geq n+\alpha/4 \Leftrightarrow \frac{\alpha}{2(2n+\alpha)} \geq \varepsilon$.

Thus we may assume that $B(x_0, Mr^{1-\varepsilon})\cap (\Gamma^+ \cup \Gamma^-) = \emptyset$. If $B(x_0, r) \subset \{u =0\}$ the almost-minimization criterion follows from the fact that $u \equiv 0$ in $B(x_0, r)$. So we may assume that $B(x_0, Mr^{1-\varepsilon}) \subset \{u > 0\}$ (the negative phase follows similarly). 

 Let $v$ minimize $J_{B(x_0,r)}$ with boundary values $u|_{\partial B(x_0, r)}$. Let us first assume that $v$ is simply a harmonic function in $B(x_0, r)$ (this will be the case when $v > 0$).  We compute that $$\int_{B(x_0, r)} |\nabla u|^2 - |\nabla v|^2 = -\int_{B(x_0,r)} |\nabla (u-v)|^2 = \int_{B(x_0,r)} \Delta u (u-v).$$ We note that 
 \[
 u-v <  
 \max_{B(x_0, r)} u - \min_{B(x_0, r)} v= \max_{B(x_0, r)} u - \min_{\partial B(x_0, r)} u 
% \min_{B(x_0, r)} u - \max_{B(x_0, r)} v= \min_{B(x_0, r)} u - \max_{\partial B(x_0, r)} u 
 < \mathrm{osc}_{B(x_0, r)} u < Cr,\] where we used the maximum principle and the assumption on $|\nabla u|$.  Similarly with a lower bound so that $\sup_{B(x_0,r)} |u-v| \leq Cr$. By our assumptions $$|\Delta u(z)| \leq \|\nabla^2 u(z)\| \leq C \mathrm{dist}(z, \{u = 0\})^{-1} \leq Cr^{\varepsilon-1},\qquad \forall z\in B(x_0,r).$$ 
 Note the last inequality above is because $B(x_0, Mr^{1-\varepsilon}) \subset \{u > 0\}$.  Putting all of this together we get that \begin{equation}\label{e:harmv}\int_{B(x_0, r)} |\nabla u|^2 - |\nabla v|^2 \leq Cr^{n+\varepsilon}.\end{equation}
 
If $v$ does vanish, the maximum principle still says that $v \geq 0$. Let $h$ be the harmonic extension of $u|_{B(x_0, r)}$ into $B:= B(x_0, r)$ and compute, recalling that $h > 0$, $$\begin{aligned} J_B(v) - J_B(h) &\geq -(\sup_B q^2_+)|\{v= 0\}\cap B| + \int_B |\nabla v|^2 - |\nabla h|^2 \\
&= -(\sup_B q^2_+)|\{v = 0\}\cap B|+ \int_B |\nabla (v-h)|^2\\ =& -(\sup_B q^2_+)|\{v = 0\}\cap B| + \int_{\partial \{v > 0\}\cap B} \partial_\nu v h\, d\mathcal H^{n-1}\\ \geq&-(\sup_B q^2_+)|\{v = 0\}\cap B| + (\inf_{B} h) (\inf_B q_+) \mathcal H^{n-1}(\partial \{v = 0\}\cap B)\\ =& -(\mathrm{osc}_Bq_+^2)|\{v=0\}\cap B|\\ +& (\inf_B q_+)\left((\inf_{\partial B} u) \mathcal H^{n-1}(\partial \{v = 0\}\cap B) - \inf_Bq_+ |\{v =0\}\cap B|\right).\end{aligned}$$

We apply the isoperimetric inequality to $\{v=0\}\cap B$ (recall that $v > 0$ on $\partial B$) and estimate $\inf_{\partial B}u \geq C^{-1}(\mathrm{dist}(x_0, \partial \{u >0\}) - r) \geq \tilde{C}^{-1}r^{1-\varepsilon}$ to get $$J_B(v) - J_B(h) \geq -Cr^{\alpha+n} + C|\{v =0\}\cap B|^{\frac{n-1}{n}}(r^{1-\varepsilon} -c|\{v =0\}\cap B|^{1-\frac{1}{n}}).$$ As long as $\varepsilon < 1/n$ and $r < r_0$ is small enough we have $J_B(v) - J_B(h) \geq -Cr^{\alpha+n}$, which combined with \eqref{e:harmv} (which compares $u$ to its harmonic extension) gives us the inequality we want.

%\geq& -(\mathrm{osc}_B q_+)r^n + (\inf_B q_+)(-|\{v=0\}| + c\delta(x_0)|\{v =0\}|^{(n-1)/n})
%\geq - Cr^{n+\alpha},\end{aligned}$$

 %where the last inequality holds if $c\delta(x_0) \geq |\{v=0\}|^{1/n}$ which holds as long as $\delta(x_0) \geq \tilde{c}r$ for some $\tilde{c}$ that depends on the isoperimetric inequality and the non-degeneracy of $u$. Since $\delta(x_0) \geq MR^{1-\varepsilon}$ we simply pick $M \geq \tilde{c}$. We can then compare $J_B(h)$ to $J_B(u)$ using \eqref{e:harmv} and we are done. 
\end{proof}

Although Lemma \ref{l:slidingbarriers}
is really a minor modification of \cite[Lemma 2.3]{DeJe09} (see also \cite{Caf1}),
we give a proof for the convenience of the reader.
\begin{proof}[Proof of Lemma \ref{l:slidingbarriers}]
Let $E  =\{\rho \in [a,b]\mid \phi_\rho \leq v\}$. Since $\rho \mapsto \phi_\rho(x_0)$ is continuous for every $x_0 \in \overline{B}$ it must be the case that $E$ is a closed subset of $[a,b]$. Furthermore $b\in E$ by assumption. Thus we will have shown that $E = [a,b]$ (and we will be done) if we can show that $E$ is (relatively) open in $[a,b]$. 

We first observe that for all $t\in E$ we have $\overline{\{\phi_t > 0\}} \cap \overline{B} \subset \overline{B} \cap \{v > 0\}$. Indeed, otherwise the zero sets of $\phi_t$ and $v$ would touch in the interior of $B$.  If $\phi_{t}(x_1) = v(x_1) = 0$ then it must be the case that $x_1 \in \partial \{v > 0\}$ (since 
$\phi_{t} \leq v$ and
$x_1 \in \partial \{\phi_{t} > 0\}$ by 
the condition (2)
on the zero set of $\phi_t$). If $x_1$ is a two-phase point for $v$ then Condition (4) implies that $\partial_{\nu^+} \phi_{t}(x_1) \geq \partial_{\nu^+} v(x_1)$ or $\partial_{\nu^-}\phi_{t}(x_1) \leq \partial_{\nu^-}v(x_1)$. Either way we get a contradiction to the Hopf maximum principle. If $x_1$ is a one-phase point for $v$ then $\partial_{\nu^+}\phi_{t}(x_1) \geq q_+(x_1) = \partial_{\nu^+}v(x_1)$ and we again get a contradiction by the Hopf maximum principle. Note in both cases we use the condition that $\phi_{t} < v$ on the boundary whenever they are both non-zero to conclude that $v-\phi_{t} \neq 0$ and thus the normal derivatives above must have a strict sign. 

Since for all $t_0\in E$ we have $\overline{\{\phi_{t_0} > 0\}} \cap \overline{B} \subset \overline{B} \cap \{v > 0\}$ and since $t\mapsto \partial \{\phi_t > 0\}$ is continuous there exists a $\varepsilon > 0$ such that if $|\tilde{t}-t_0| < \varepsilon$ we have $\overline{\{\phi_{\tilde{t}} > 0\}} \cap \overline{B} \subset \overline{B} \cap \{v > 0\}$. 

Thus in $\{\phi_{\tilde{t}} > 0\}$ we have that both functions are harmonic and $v \geq \phi_{\tilde{t}}$ on $\partial (\{\phi_{\tilde{t}} > 0\}\cap B) = (\partial B \cap \{\phi_{\tilde{t}} > 0\})\cup (\partial \{\phi_{\tilde{t}} > 0\}\cap \overline{B})$. Thus $v \geq \phi_{\tilde{t}}$ on this set. Furthermore $v > 0 = \phi_{\tilde{t}}$ on $\{\phi_{\tilde{t}} = 0\} \cap B$ (by the assumption that the zero set is the free boundary). 

Finally, on the set $\{v \leq 0\}$ we have $\phi_{\tilde{t}} < 0$ so the difference $v - \phi_{\tilde{t}}$ is superharmonic on $\{v \leq 0\}$. On the boundary of this set we have $v- \phi_{\tilde{t}} \geq 0$ so by the maximum principle we have $v \geq \phi_{\tilde{t}}$ on this set and putting this together with the previous parts on all $\overline{B}$. 
This shows
$E$ is open and we are done. 
\end{proof}

We are ready to finish by constructing the functions satisfying the hypothesis of Lemma~\ref{l:fbam}. First we recall the definition of the regularized distance function, introduced in \cite{DaFeMa19}:

\begin{equation}\label{e:regdist}
D_{\mu, \beta}(x) := \left(\int \frac{1}{|x-y|^{d+\beta}}d\mu(y)\right)^{-1/\beta}.
\end{equation}

We use the following facts about this regularized distance, see \cite{DaFeMa19, DaEnMa21}:

\begin{lemma}\label{l:regdistproperties}
Let $\beta > 0$ and let $\mu$ be a $d$-Ahlfors regular measure. Then $$D_{\mu,\beta} \in C^\infty(\mathbb R^n \backslash \mathrm{spt} \mu)\cap C(\mathbb R^n)$$ satisfies the following estimates with $C_1 > 0$ depending on $n, \beta, d$ and the Ahlfors regularity character of $\mu$: \begin{equation}\label{e:rdist2}\begin{aligned} C_1^{-1}\mathrm{dist}(x, \mathrm{spt}\mu) \leq D_{\mu, \beta}(x) \leq& C_1\mathrm{dist}(x, \mathrm{spt} \mu)\\
\|\nabla D_{\mu, \beta}(x) \| \leq& C_1\\
\|D^2 D_{\mu, \beta}(x) \| \leq& C_1\mathrm{dist}(x, \mathrm{spt} \mu)^{-1}.\end{aligned}\end{equation}

Furthermore, there is some dimensional constant $c = c(n, d,\beta) > 0$ such that if $\mathrm{spt} \mu$ is given by a $C^2$-submanifold and $\Theta^d(\mu, Q) := \mathrm{lim}_{r\downarrow 0} \frac{\mu(B(Q,r))}{r^d}$ satisfies $\Theta^d(\mu, Q)^{-1/\beta} \in C^{0,\alpha}(\mathrm{spt}(\mu))$ then $|\nabla D_{\mu,\beta}| \in C^{0,\alpha}(\mathbb R^n)$ and on the support of $\mu$ we have $|\nabla D_{\mu, \beta}| = c \Theta^d(\mu, Q)^{-1/\beta}$.
\end{lemma}

\begin{proof}
The estimates in \eqref{e:rdist2} are straightfoward applications of the Ahlfors regularity condition (c.f. \cite{DaFeMa19, DaEnMa21}). If $\mathrm{spt}\mu$ has a tangent at $x_0$ and if $\Theta^d(\mu, x_0)$ exists, the fact that the non-tangential limit of $|\nabla D_{\mu, \beta}|$ exists at $x_0$ and is equal to $c\Theta^d(\mu, Q)^{-1/\beta}$ is contained in the proof of \cite[Theorem 5.3]{DaEnMa21}.

Thus to complete the proof of the lemma, it suffices to show that $|\nabla D_{\mu, \beta}|$ extends to $\mathrm{spt}(\mu)$ in a H\"older continuous fashion. This follows from the $C^2$-character of $\mathrm{spt}(\mu)$ and the estimates in \cite[Section 2]{DaEnMa21}; to be slightly more precise the difference between $\nabla D_{\mu, \beta}$ at a point $x \in \mathbb R^n \setminus \mathrm{spt}(\mu)$ and the closest point $Q \in \mathrm{spt}(\mu)$ is controlled by the $\alpha$ numbers (in the sense of Tolsa \cite{Tol09}) of $\mu$ at the point $Q$ and the scale $|x-Q|$ (c.f. \cite[equation (2.19)]{DaEnMa21}). Since, $\Theta^d(\mu, Q)$ and $\mathrm{spt}(\mu)$ are regular, these $\alpha$-numbers decay like $r^\alpha$, which gives the desired result. 
\end{proof}

We are now ready to compute the proof of Theorem \ref{t:fbam}:

\begin{proof}[Proof of Theorem \ref{t:fbam}]
Let $d\mu^{\pm}(Q) = q_{\pm}^{-1}(Q) d\mathcal H^{n-1}|_{\Gamma^{\pm}}(Q)$ and let $u^{\pm} = D_{\mu^{\pm}, 1}$ in $\Omega^{\pm}$ and identically equal to zero elsewhere. Note that $\mu^{\pm}$ are Ahlfors regular by the boundedness and non-degeneracy of $q_{\pm}$ and the $C^2$-character of $f^{\pm}$ (we don't need to worry about issues at infinity because $f^{\pm}$ are constant outside of a large ball). 

The result immediately follows from Lemmas \ref{l:regdistproperties} and \ref{l:fbam}.
\end{proof}

\section{Appendix}

\subsection{Proof of Lemma \ref{slicedmin}}
Fix $-3N\le x \le 3N$. In this section we will study minimizers of \eqref{slice energy} with boundary data $v_N(\pm 1)=\pm f_N(x)$.

We know that the set $\{v_N=0\}$ is an interval, and that $v_N$ is harmonic 
(i.e. affine)
in 
$\{v_N\neq 0\}\cap (-1,1)$.
Hence we can search for $v_N$ among 
the function $v_N$ with boundary data $v_N(\pm 1)=\pm f_N(x)$, with $v_N=0$ on $[a,b]$, 
and which are % is
linear on $[-1,a]$ and $[b,1]$:

\begin{align*}
v_N(x)=
\begin{cases}
\frac{f_N(x)}{a+1}(y-a), &\text{ if } -1\le y\le a\\
0, & \text{ if } a\le y\le b\\
\frac{f_N(x)}{1-b}(y-b), &\text{ if } b\le y\le 1.
\end{cases}
\end{align*}

Then 
\[
H(v_N)=\frac{(f_N(x))^2}{1-b}+\frac{(f_N(x))^2}{a+1}+2-(b-a).
\]

We will minimize $G(a,b)=\frac{(f_N(x))^2}{1-b}+\frac{(f_N(x))^2}{a+1}+2-(b-a)$ 
given the constraint $-1<a\le b<1$.

\subsubsection{Assume $f_N(x)\ge 1$}\label{SS>1}
We claim that $G(0,0)\le G(a,b)$ for any $-1<a\le b<1$, hence the minimizer is $v_N(x)=yf_N(x)$.

Notice that $G(0,0)=2(f_N(x))^2+2$, hence it suffices to show
\[
2(f_N(x))^2+2\le (f_N(x))^2\left(\frac{1}{1-b}+\frac{1}{a+1}\right)+2-b+a.
\]
This is equivalent to
\[
b-a\le (f_N(x))^2\left(\frac{1}{1-b}-1 +\frac{1}{a+1}-1\right)=(f_N(x))^2\left(\frac{b}{1-b}-\frac{a}{a+1}\right)
\]
Since $f_N(x)\ge 1$, it suffices to show that \[
b-a\le \frac{b}{1-b}-\frac{a}{a+1}
\]
which holds since $-1<a$ and $b<1$.

\subsubsection{Assume $0\le f_N(x)< 1$}\label{SS<1}
Looking for critical points of $G$, one looks for solutions of 
\[
\frac{\partial G}{\partial a}(a,b)=-\frac{(f_N(x))^2}{(1+a)^2}+1=0, \ 
\frac{\partial G}{\partial b}(a,b)=\frac{(f_N(x))^2}{(1-b)^2}-1=0,
\]
%We obtain $1+a=\pm f(x)$, hence $a=-1\pm f(x)$. Since $0<f(x)<2$ and we need $-1<a$, that leads to $a=f(x)-1$. Similarly, the other equality gives us $1-b=\pm f(x)$ or $b=1\pm f(x)$. Since we need $b<1$ and $0<f(x)$, that leads us to $b=1-f(x)$.
leading to 
$(1+a)^2 = (1-b)^2 = (f_N(x))^2$. Now $1+a = -f_N(x)$ and $1-b = -f_N(x)$
are impossible (unless $f_N(x)=0$) because $-1<a \leq b<1$, so 
$a=f_N(x)-1$ and $b=1-f_N(x)$.
Notice that $a\le b$ is only satisfied when $f_N(x)\le 1$.

We have $G(f_N(x)-1,1-f_N(x))=4f_N(x)$. Let us now show $4f_N(x)\le G(a,b)$ for any $-1<a\le b<1$.
It suffices to show
\[
2f_N(x)\le \frac{(f_N(x))^2}{1-b}+1-b, \ \ \text{ and } \  \ 2f_N(x)\le \frac{(f_N(x))^2}{a+1}+1+a,
\]
which are both true, as the minimizers of the right-hand sides occur when $b=1-f_N(x)$ and $a=f_N(x)-1$.

%The first inequality is equivalent to $2f(x)(1-b)\le (f(x))^2+(1-b)^2$. Calling $1-b=w$, this is equivalent to $-w^2+2f(x)w-(f(x))^2\le 0$. We see $-w^2+2f(x)w-(f(x))^2$ is a parabola opening down, whose only root is $w=f(x)$, or $b=1-f(x)$. The second inequality is equivalent to $2f(x)(a+1)\le (f(x))^2+(1+a)^2$. Calling $1+a=w$, this is again equivalent to $-w^2+2f(x)w-(f(x))^2\le 0$. Once again, we obtain a parabola opening down, whose only root is $w=f(x)$, or $a=f(x)-1$.

Hence the minimizer is $v_N(y)=\text{sgn}(y)(|y|-1+f_N(x))_+.$

\subsection{Energy bounds when $x\in [-N, N]$} \label{SS:X2}
In this subsection we want to provide some crucial bounds on $H_x(u_N(x, -))$ when $x\in [-N, N]$. Notice that if $x\in [-N, N]$, $v_N(x,y)=\text{sgn}(y)(|y|-\alpha)_+$. Therefore $H_x(v_N(x,\cdot))=4(1-\alpha)$. We want to show that % is it OK ?
if $u_N$ is large and % added and
$v_N$ is zero, then the slice energy of $u_N$ is quantitatively bigger than that of $v_N$.

More precisely we have the following claim:

\begin{claim}\label{cl:energybound}
Let $x\in [-N, N]\backslash X_0$ and assume that $u_N(x,y) > \beta > 0$ (or $u_N(x,y) < -\beta < 0$) for some $y < \alpha/2$. Then there exists $\eta = \eta(\alpha, \beta) > 0$ such that $H(u_N(x,\cdot))\ge H(v_N(x,\cdot))+\eta$.
\end{claim}

We first start with estimates on the energy depending on the size of the zero set of $u_N$: call $A=\{y\in [-1,1] \ : \ u_N(x,y)=0\}$ and $\delta= |A|$, so $0\le \delta< 2$.
We have
\[
H(u_N(x,\cdot))=2-\delta+\int_{-1}^1|\partial_yu_N(x,\cdot)|^2dy
\]
Now, since $f_N(x)=1-\alpha$ for $|x|<N$, by comparing with a linear function we have
\begin{align*}
    \int_{-1}^1|\partial_yu_N(x,\cdot)|^2dy&\ge \frac{(1-\alpha)^2}{\inf A +1}+\frac{(1-\alpha)^2}{1-\sup A}=(1-\alpha)^2
    \left(\frac{1}{1-\sup A}+\frac{1}{\inf A+1}\right).
\end{align*}
Recall that for $x,y>0$, $\frac{1}{x}+\frac{1}{y}\ge \frac{4}{x+y}$, with equality only if $x=y$. Applying this inequality with $x=1-\sup A$ and $y=\inf A+1$ we conclude
\begin{align*}
\int_{-1}^1|\partial_yu_N(x,\cdot)|^2dy&\ge (1-\alpha)^2\frac{4}{2-\sup A+\inf A}=\frac{4(1-\alpha)^2}{2-(\sup A-\inf A)}
\ge \frac{4(1-\alpha)^2}{2-|A|}\\
&=\frac{4(1-\alpha)^2}{2-\delta}.
\end{align*}
Therefore
\[
H(u_N(x,\cdot))\ge 2-\delta+\frac{4(1-\alpha)^2}{2-\delta} = 4(1-\alpha) + \frac{(2(1-\alpha) - (2-\delta))^2}{2-\delta}.
\]
Rearranging we get by Lemma \ref{slicedmin}:

\begin{equation}\label{e:energy}
    H(u_N(x,\cdot)) \geq H(v_N(x, \cdot)) + \frac{(\delta - 2\alpha)^2}{2-\delta}.
\end{equation}
Notice that if $\delta\le \alpha/2$, \eqref{e:energy} gives $H(u_N(x,\cdot))\ge H(v_N(x,\cdot))+\frac{\alpha^2}{2}$, hence we can assume $\delta > \alpha/2$.

Now we assume there is a point $(x, y_0)$ with $y_0 < \alpha/2$ and $u_N(x, y_0) > \beta > 0$ (the case where it is negative follows similarly). We have two cases:

\noindent {\bf Case 1:} Assume there exists $\tilde{y} > y_0$ with $u_N(x, \tilde{y}) = 0$. Let $A_+ = \sup \{y \mid u_N(x, y) = 0\}$, be the largest such value of $y$, and let $A_- = \inf\{y\mid y > y_0, u_N(x, y) = 0\}$ be the smallest such value of $y$. Furthermore let $B_+ = \sup\{y < \tilde{y}\mid u_N(x,y) = 0\}$ and $B_- = \inf\{y\mid u_N(x,y) = 0\}$. We then have that $$H(u_N(x,y)) \geq 2- (A_+ - A_- + B_+ - B_-)+ \frac{\beta^2}{A_--\tilde{y}} + \frac{\beta^2}{\tilde{y} - B_+} + (1-\alpha)^2 \left(\frac{1}{B_- + 1} + \frac{1}{1-A_+}\right).$$ Using again $1/x + 1/y \geq \frac{4}{x+y}$ we have $$H(u_N(x,y)) \geq 2-(A_+ - A_- + B_+ - B_-) + \frac{4\beta^2}{A_--B_+} + \frac{4(1-\alpha)^2}{2-(A_+-B_-)}.$$ Let $z = (A_- - B_+)$ and $w= A_+ - B_-$. We notice that $w-z = A_+ - A_- + B_+-B_- \geq \delta$. Let $G(z,w) := 2-(w-z) + \frac{4\beta^2}{z} + \frac{4(1-\alpha)^2}{2-w}$. We want to find the minimum of $G(z,w)$ on the domain $(z,w) \in [0,2]^2 \cap\{w-z \geq \delta\}$ as this gives a lower bound for $H(u_N(x,y))$. Straight forward calculus shows that the only critical point of $G(z,w)$ is a local minimum at $(z,w) = (2\beta, 2\alpha)$ which gives the lower bound $H(u_N(x,-)) \geq 4(1-\alpha) + 4\beta = H(v_N(x,-)) + 4\beta$. On the boundary of the domain, it is easy to see that $G(z,w) = +\infty$ when $z =0$ or $w = 2$. When $w-z = \delta$ we rewrite the offset equation above, $$\begin{aligned} H(u_N(x,y)) \geq& 2-\delta + \frac{4\beta^2}{z} + \frac{4(1-\alpha)^2}{2-(z+\delta)}\\
\geq& 2-\delta -z + z+ \frac{4\beta^2}{z} + \frac{4(1-\alpha)^2}{2-(z+\delta)}\\
 \geq& 4(1-\alpha) + z + \frac{4\beta^2}{z}\\ +& \frac{-4(1-\alpha)(2-\delta-z) + (2-\delta-z)^2 + 4(1-\alpha)^2}{2-\delta-z}\\
\geq& H(v_N(x,-)) + \frac{4\beta^2}{z} + \frac{(\delta + z -2\alpha)^2}{2-\delta-z} +z\\
\geq& H(v_N(x,-)) + \frac{4\beta^2}{z} + z.\end{aligned}$$

Minimizing this in $z$ we get that $z = 2\beta$, giving us the same lower bound as above, $H(u_N(x,-)) \geq H(v_N(x,-)) + 4\beta$.

\medskip

\noindent {\bf Case 2:} Here we have that the zero set of $u_N(x, -)$ lies totally below $y_0$, i.e. if $A = \{y\mid u_N(x,y) = 0\}$, then $\sup A < y_0$. Calculating as above we have
$$H(u_N(x,-)) \geq 2-\delta + (1-\alpha)^2\left(\frac{1}{\inf A + 1} + \frac{1}{1-\sup A}\right).$$ 

Note that $\inf A\le \sup A \leq y_0 \leq \alpha/2$ and $\sup A-\inf A\ge |A|=\delta$.

Calling $x=\inf A$ and $y=\sup A$ we obtain the bounds $x+\delta\le y\le \frac{\alpha}{2},$ $-1<x\le \frac{\alpha}{2}-\delta$. We want to minimize the function $G(a,b)=\frac{1}{a+1}+\frac{1}{1-b}$ in the domain given by those bounds, recalling that $\frac{\alpha}{2}<\delta$. Notice that as $a$ approaches $-1$, $G(a,b)$ goes to infinity. When $b=\frac{\alpha}{2}$, we minimize $G(a,b)$ when $a=\frac{\alpha}{2}-\delta$. Finally, when $b=a+\delta$, we obtain the function $g(a)=\frac{1}{a+1}+\frac{1}{1-a-\delta}$, with $-1<a\le \frac{\alpha}{2}-\delta$. This function has a minimum at $a=-\frac{\delta}{2}$.

Case a: when  $\frac{\alpha}{2}-\delta<-\frac{\delta}{2}$ (that is, $\frac{\delta}{2}<\frac{\alpha}{2}$), one compares $G(\alpha/2-\delta,\alpha/2)=\frac{1}{1+\frac{\alpha}{2}-\delta}+\frac{1}{1-\frac{\alpha}{2}}$ with $G(-\delta/2, \delta/2)=\frac{2}{1-\frac{\delta}{2}}$ to find 
\[
G(a,b)\ge G(-\delta/2,\delta/2)=\frac{4}{2-\delta}.
\]
In this case we follow the computations of \ref{e:energy}, recalling we are now under the case $\delta<\alpha$, and notice that
\[
H(u_N(x,\cdot))\ge H(v_N(x,\cdot))+\frac{(\delta-2\alpha)^2}{2-\delta}\ge H(v_N(x,\cdot))+\frac{(\delta-2\alpha)^2}{2-\delta}\ge H(v_N(x,\cdot))+\frac{\alpha^2}{2}.
\]

Case b: when $\frac{\alpha}{2}-\delta\le -\frac{\delta}{2}$ (that is, $\frac{\alpha}{2}\le \frac{\delta}{2}$), we conclude the minimum of $G$ is attained when $a=\inf A=\frac{\alpha}{2}-\delta$ and $a=\sup A=\frac{\alpha}{2}$.
Note that $\frac{1}{w} + \frac{1}{z} = \frac{4}{w+z} + \frac{(w-z)^2}{wz(w+z)}$ and apply this to $w = 1-\alpha/2, z= 1+\alpha/2-\delta$. Using the same arguments as above we get that, 
$$H(u_N(x, -)) \geq H(v_N(x,-)) + \frac{(\delta - 2\alpha)^2}{2-\delta} + \frac{(\alpha-\delta)^2}{(1-\delta)(1-\alpha/2)(1+\alpha/2-\delta)}.$$ 
A straightforward calculation shows that $H(u_N(x,-)) \geq H(v_N(x,-)) + \frac{\alpha^2}{10^4}$.

\bibliography{pools}{}

\providecommand{\bysame}{\leavevmode\hbox to3em{\hrulefill}\thinspace}
\providecommand{\MR}{\relax\ifhmode\unskip\space\fi MR }
\providecommand{\MRhref}[2]{%
  \href{http://www.ams.org/mathscinet-getitem?mr=#1}{#2}
}
\providecommand{\href}[2]{#2}
\begin{thebibliography}{DESVGT21}

\bibitem[AC81]{AlCa81}
H.~W. Alt and L.~A. Caffarelli, \emph{Existence and regularity for a minimum
  problem with free boundary}, J. Reine Angew. Math. \textbf{325} (1981),
  105--144. \MR{618549}

\bibitem[ACF84]{AlCaFr84}
Hans~Wilhelm Alt, Luis~A. Caffarelli, and Avner Friedman, \emph{Variational
  problems with two phases and their free boundaries}, Trans. Amer. Math. Soc.
  \textbf{282} (1984), no.~2, 431--461. \MR{732100}

\bibitem[Caf87]{Caf1}
Luis~A. Caffarelli, \emph{A {H}arnack inequality approach to the regularity of
  free boundaries. {I}. {L}ipschitz free boundaries are {$C^{1,\alpha}$}}, Rev.
  Mat. Iberoamericana \textbf{3} (1987), no.~2, 139--162. \MR{990856}

\bibitem[Caf88]{Caf3}
\bysame, \emph{A {H}arnack inequality approach to the regularity of free
  boundaries. {III}. {E}xistence theory, compactness, and dependence on {$X$}},
  Ann. Scuola Norm. Sup. Pisa Cl. Sci. (4) \textbf{15} (1988), no.~4, 583--602
  (1989). \MR{1029856}

\bibitem[Caf89]{Caf2}
\bysame, \emph{A {H}arnack inequality approach to the regularity of free
  boundaries. {II}. {F}lat free boundaries are {L}ipschitz}, Comm. Pure Appl.
  Math. \textbf{42} (1989), no.~1, 55--78. \MR{973745}

\bibitem[Cha88]{Ch88}
Sheldon Xu-Dong Chang, \emph{Two-dimensional area minimizing integral currents
  are classical minimal surfaces}, J. Amer. Math. Soc. \textbf{1} (1988),
  no.~4, 699--778. \MR{946554}

\bibitem[CLS19]{ChSa19}
H\'{e}ctor Chang-Lara and Ovidiu Savin, \emph{Boundary regularity for the free
  boundary in the one-phase problem}, New developments in the analysis of
  nonlocal operators, Contemp. Math., vol. 723, Amer. Math. Soc., [Providence],
  RI, 2019, pp.~149--165. \MR{3916702}

\bibitem[CS05]{CafBook}
Luis Caffarelli and Sandro Salsa, \emph{A geometric approach to free boundary
  problems}, Graduate Studies in Mathematics, vol.~68, American Mathematical
  Society, Providence, RI, 2005. \MR{2145284}

\bibitem[DEM21]{DaEnMa21}
Guy David, Max Engelstein, and Svitlana Mayboroda, \emph{Square functions,
  nontangential limits, and harmonic measure in codimension larger than 1},
  Duke Math. J. \textbf{170} (2021), no.~3, 455--501. \MR{4255042}

\bibitem[DESVGT21]{DaEnSmTo21}
Guy David, Max Engelstein, Mariana Smit Vega~Garcia, and Tatiana Toro,
  \emph{Regularity for almost-minimizers of variable coefficient
  {B}ernoulli-type functionals}, Math. Z. \textbf{299} (2021), no.~3-4,
  2131--2169. \MR{4329282}

\bibitem[DET19]{DaEnTo19}
Guy David, Max Engelstein, and Tatiana Toro, \emph{Free boundary regularity for
  almost-minimizers}, Adv. Math. \textbf{350} (2019), 1109--1192. \MR{3948692}

\bibitem[DFM19]{DaFeMa19}
Guy David, Joseph Feneuil, and Svitlana Mayboroda, \emph{Dahlberg's theorem in
  higher co-dimension}, J. Funct. Anal. \textbf{276} (2019), no.~9, 2731--2820.
  \MR{3926132}

\bibitem[DJ90]{DaJe90}
G.~David and D.~Jerison, \emph{Lipschitz approximation to hypersurfaces,
  harmonic measure, and singular integrals}, Indiana Univ. Math. J. \textbf{39}
  (1990), no.~3, 831--845. \MR{1078740}

\bibitem[DLMSV18]{DeMaSpVa18}
Camillo De~Lellis, Andrea Marchese, Emanuele Spadaro, and Daniele Valtorta,
  \emph{Rectifiability and upper {M}inkowski bounds for singularities of
  harmonic {$Q$}-valued maps}, Comment. Math. Helv. \textbf{93} (2018), no.~4,
  737--779. \MR{3880226}

\bibitem[DLSS17]{DeSpSp17}
Camillo De~Lellis, Emanuele Spadaro, and Luca Spolaor, \emph{Regularity theory
  for 2-dimensional almost minimal currents {II}: {B}ranched center manifold},
  Ann. PDE \textbf{3} (2017), no.~2, Paper No. 18, 85. \MR{3712561}

\bibitem[DLSS18]{DeSpSp18}
\bysame, \emph{Regularity theory for {$2$}-dimensional almost minimal currents
  {I}: {L}ipschitz approximation}, Trans. Amer. Math. Soc. \textbf{370} (2018),
  no.~3, 1783--1801. \MR{3739191}

\bibitem[DLSS20]{DeSpSp20}
\bysame, \emph{Regularity theory for 2-dimensional almost minimal currents
  {III}: {B}lowup}, J. Differential Geom. \textbf{116} (2020), no.~1, 125--185.
  \MR{4146358}

\bibitem[DPSV21a]{DeSpVePreprint}
Guido De~Philippis, Luca Spolaor, and Bozhidar Velichkov,
  \emph{(quasi-)conformal methods in two-dimensional free boundary problems},
  2021.

\bibitem[DPSV21b]{DeSpVe21}
\bysame, \emph{Regularity of the free boundary for the two-phase {B}ernoulli
  problem}, Invent. Math. \textbf{225} (2021), no.~2, 347--394. \MR{4285137}

\bibitem[DSFS14a]{DeFeSa14}
Daniela De~Silva, Fausto Ferrari, and Sandro Salsa, \emph{On two phase free
  boundary problems governed by elliptic equations with distributed sources},
  Discrete Contin. Dyn. Syst. Ser. S \textbf{7} (2014), no.~4, 673--693.
  \MR{3177758}

\bibitem[DSFS14b]{DeFeSaReg}
\bysame, \emph{Two-phase problems with distributed sources: regularity of the
  free boundary}, Anal. PDE \textbf{7} (2014), no.~2, 267--310. \MR{3218810}

\bibitem[DSFS19]{DeFeSaSurvey}
\bysame, \emph{Recent progresses on elliptic two-phase free boundary problems},
  Discrete Contin. Dyn. Syst. \textbf{39} (2019), no.~12, 6961--6978.
  \MR{4026176}

\bibitem[DSJ09]{DeJe09}
Daniela De~Silva and David Jerison, \emph{A singular energy minimizing free
  boundary}, J. Reine Angew. Math. \textbf{635} (2009), 1--21. \MR{2572253}

\bibitem[DT15]{DaTo15}
G.~David and T.~Toro, \emph{Regularity of almost minimizers with free
  boundary}, Calc. Var. Partial Differential Equations \textbf{54} (2015),
  no.~1, 455--524. \MR{3385167}

\bibitem[JK82]{JeKe82}
David~S. Jerison and Carlos~E. Kenig, \emph{Boundary behavior of harmonic
  functions in nontangentially accessible domains}, Adv. in Math. \textbf{46}
  (1982), no.~1, 80--147. \MR{676988}

\bibitem[MTV17]{MaTeVe17}
Dario Mazzoleni, Susanna Terracini, and Bozhidar Velichkov, \emph{Regularity of
  the optimal sets for some spectral functionals}, Geom. Funct. Anal.
  \textbf{27} (2017), no.~2, 373--426. \MR{3626615}

\bibitem[Ste70]{Steinbook}
Elias~M. Stein, \emph{Singular integrals and differentiability properties of
  functions}, Princeton Mathematical Series, No. 30, Princeton University
  Press, Princeton, N.J., 1970. \MR{0290095}

\bibitem[SV19]{SpVe19}
Luca Spolaor and Bozhidar Velichkov, \emph{An epiperimetric inequality for the
  regularity of some free boundary problems: the 2-dimensional case}, Comm.
  Pure Appl. Math. \textbf{72} (2019), no.~2, 375--421. \MR{3896024}

\bibitem[Tol09]{Tol09}
Xavier Tolsa, \emph{Uniform rectifiability, {C}alder\'{o}n-{Z}ygmund operators
  with odd kernel, and quasiorthogonality}, Proc. Lond. Math. Soc. (3)
  \textbf{98} (2009), no.~2, 393--426. \MR{2481953}

\bibitem[Wei99]{Wei99}
Georg~Sebastian Weiss, \emph{Partial regularity for a minimum problem with free
  boundary}, J. Geom. Anal. \textbf{9} (1999), no.~2, 317--326. \MR{1759450}

\end{thebibliography}
\bibliographystyle{amsalpha}

%\printbibliography
%\bibliographystyle{amsalpha}

\end{document}